\documentclass[11pt,reqno]{amsart}
\usepackage[T1]{fontenc}
\usepackage{lmodern}
\usepackage{amsmath,amssymb,mathtools,mathrsfs}
\usepackage{microtype}
\usepackage{tikz-cd}
\usepackage{booktabs}
\usepackage{xurl}
\usepackage[colorlinks=true,linkcolor=blue,citecolor=blue,urlcolor=blue]{hyperref}
\hypersetup{
  pdftitle={Rigid and Obstructed Tangent Bundles on Calabi-Yau Threefolds},
  pdfauthor={Xueyuan Wan},
  pdfsubject={Stable tangent-bundle rigidity, a counterexample to Peternell, and an explicit Kuranishi space},
  pdfkeywords={Calabi-Yau threefold, tangent bundle, deformation, Kuranishi space, finite quotient}}
\providecommand{\doi}[1]{\href{https://doi.org/#1}{\nolinkurl{#1}}}
\newcommand{\CC}{\mathbb C}
\newcommand{\PP}{\mathbb P}
\newcommand{\ZZ}{\mathbb Z}
\newcommand{\FF}{\mathbb F}
\newcommand{\OO}{\mathcal O}
\DeclareMathOperator{\End}{End}

\DeclareMathOperator{\Ext}{Ext}
\DeclareMathOperator{\Def}{Def}
\DeclareMathOperator{\Aut}{Aut}

\DeclareMathOperator{\Sym}{Sym}
\newtheorem{theorem}{Theorem}[section]
\newtheorem{proposition}[theorem]{Proposition}
\newtheorem{lemma}[theorem]{Lemma}
\newtheorem{question}[theorem]{Question}
\newtheorem{corollary}[theorem]{Corollary}
\theoremstyle{definition}
\newtheorem{convention}[theorem]{Convention}

\theoremstyle{remark}
\newtheorem{remark}[theorem]{Remark}

\numberwithin{equation}{section}
\allowdisplaybreaks[1]
\title[Rigid and obstructed tangent bundles]{Rigid and Obstructed Tangent Bundles\texorpdfstring{\\}{ }on Calabi--Yau Threefolds}
\date{}
\makeatletter
\@namedef{subjclassname@2020}{\textup{2020} Mathematics Subject Classification}
\makeatother
\subjclass[2020]{14J32, 14J60, 32G08 }
\keywords{Calabi--Yau threefold, tangent bundle, bundle deformation,
finite \'etale quotient, rigidity, Kuranishi space}

\begin{document}
\author{Xueyuan Wan}
\address{Mathematical Science Research Center, Chongqing University of Technology,
Chongqing 400054, China}
\email{xwan@cqut.edu.cn}
\begin{abstract}
We address two questions of Huybrechts concerning rigidity and singular deformation spaces of tangent bundles on Calabi--Yau threefolds, without assuming simple connectedness. Our first example is a free $(\mathbb Z/2)^3$ quotient of a smooth intersection of four quadrics in $\mathbb P^7$. Its tangent bundle is stable and infinitesimally rigid, and it also provides a three-dimensional counterexample to the proposed classification in Peternell's Question~1.6. For a classical Igusa quotient of a product of three elliptic curves, we determine the analytic semiuniversal deformation germ of its polystable tangent bundle, keeping the underlying threefold fixed. This germ is isomorphic to the product of two copies of the union of the three coordinate axes in $\mathbb C^3$. It is therefore reduced and singular, with nine two-dimensional irreducible components.
\end{abstract}
\maketitle

\section{Introduction}

Deforming the tangent bundle of a complex manifold while keeping the
manifold fixed is different from deforming its complex structure.  The two
problems have tangent spaces $H^1(X,\End T_X)$ and $H^1(X,T_X)$,
respectively.  In particular, the unobstructedness of deformations of a
Calabi--Yau manifold \cite{Tian1987,Todorov1989} does not imply
unobstructedness for its tangent bundle.

The study of $H^1(X,\End T_X)$ was developed in early calculations of
endomorphism-valued cohomology motivated by superstring compactifications
\cite{DGKM1989,EastwoodHubsch1990}.  In
\cite[\S2]{Huybrechts1995}, Huybrechts singled out three questions:
\begin{enumerate}
\item[(i)] Can a Calabi--Yau threefold have an infinitesimally rigid
  tangent bundle, that is, $H^1(X,\End T_X)=0$?
\item[(ii)] Must $h^1(X,\End T_X)$ remain constant when the complex
  structure of $X$ varies?
\item[(iii)] Can the deformation space of $T_X$, with $X$ fixed, be singular?
\end{enumerate}
Question~(ii) has a negative answer: Aspinwall, Melnikov,
and Plesser \cite[\S5.3.1 and Appendix~A.2]{AMP2012} computed
$h^1(\End T)=200$ at the Fermat resolved octic and $188$ at a general
member of the same smooth family. We concentrate on questions~(i)
and~(iii).

\begin{convention}\label{conv:cy}
A Calabi--Yau threefold is a smooth connected projective variety $X$ over
$\CC$ with $K_X\simeq\OO_X$ and $H^1(X,\OO_X)=0$.  Simple connectivity
is not required.  For a holomorphic vector bundle $E$ on $X$, the notation
$\Def_X(E)$ denotes a germ of a semiuniversal analytic deformation base
with $X$ fixed.  We call $E$ \emph{infinitesimally rigid}, or simply
\emph{rigid}, when $H^1(X,\End E)=0$.
\end{convention}

We follow the convention in
\cite[Introduction, Notation]{Huybrechts1995},
with the additional assumption that $X$ is projective.
Simple connectivity is not required.  The absence of a simple-connectivity
assumption is essential to both constructions below.

\subsection{Rigidity and Peternell's question}
Our first result identifies a stable rigid tangent bundle on a classical
free quotient of a complete intersection.

\begin{theorem}\label{thm:rigid}
There is a smooth intersection $Y\subset\PP^7$ of four diagonal quadrics
with a free action of $G\simeq(\ZZ/2)^3$ such that $X=Y/G$ is a
Calabi--Yau threefold and $T_X$ is slope-stable with respect to every ample
polarization.  Moreover,
\[
H^1(X,\End T_X)=0,\quad
\pi_1(X)\simeq(\ZZ/2)^3,\quad
(h^{1,1}(X),h^{2,1}(X))=(1,9).
\]
\end{theorem}

The relevance of this example extends beyond Calabi--Yau geometry.
Peternell \cite[Notation~1.1 and Remark~1.2]{Peternell2020} calls
a first-order deformation of $T_X$ \emph{genuine} if it is not
obtained by tensoring $T_X$ with a first-order deformation of
$\OO_X$.  Equivalently, genuine first-order deformations exist
precisely when
\[
H^1(X,\End_0 T_X)\ne 0,
\]
where $\End_0 T_X$ denotes the sheaf of trace-free endomorphisms.

For a compact K\"ahler manifold $X$, Peternell formulates the
proposed characterization in \cite[Question~1.6]{Peternell2020}
as follows:
\begin{question}[Peternell]\label{question}
For a compact K\"ahler manifold $X$, Then
$T_X$ has no genuine first order deformation if and only if $X$ is one of the
following.
\begin{itemize}
\item $X \simeq \mathbb{P}_n$
\item $X$ is a ball quotient of dimension at least three
\item $X$ is a two-dimensional ball quotient with $H^1(X,\End_0 T_X)=0$
\item $X$ is a product (with at least two factors)
$$
X=\prod_j X_j
$$
with $X_j$ ball quotients such that $q(X_j)=0$ and
$H^1(X_j,\End_0 T_{X_j})=0$ for all $j$.
\end{itemize}	
\end{question}

Theorem~\ref{thm:rigid} provides a counterexample to the
``only if'' direction of this proposed characterization.

\begin{corollary}
\label{cor:peternell}
There exists a smooth projective Calabi--Yau threefold $X$ whose
tangent bundle is stable and has no genuine first-order
deformations, but which belongs to none of the four classes
listed in Question \ref{question}.
In particular, the proposed characterization fails in complex
dimension three.
\end{corollary}

\begin{proof}
Take the threefold $X$ of Theorem~\ref{thm:rigid}. It is compact
K\"ahler, and the trace splitting
\[
\End T_X\simeq
\OO_X\cdot\operatorname{id}_{T_X}\oplus\End_0 T_X
\]
shows that $H^1(X,\End_0 T_X)=0$. Thus $T_X$ has no genuine
first-order deformations.

Its fundamental group is $(\mathbb Z/2)^3$, so $X$ is not
$\mathbb P^3$. A compact ball quotient of positive dimension
has infinite fundamental group: otherwise its noncompact
universal cover would be a finite cover of a compact space.
The same holds for a product of positive-dimensional compact
ball quotients. Thus $X$ belongs to none of the four listed classes.
\end{proof}

The geometry of the free quotient is not new.  Free $(\ZZ/2)^3$
actions on intersections of four quadrics occur in
\cite[pp.~65, 67]{CHSW1985} and
\cite[p.~343]{StromingerWitten1985}; Hua studies and classifies the relevant
free actions in \cite{Hua2011}.  Our point is the vanishing of the entire
first-order tangent-bundle deformation space on such a quotient.  If
$W=H^0(Y,\OO_Y(1))$ and $U\subset\Sym^2W$ is the space of defining
quadrics, we obtain an equivariant isomorphism
$$
H^1(Y,\End T_Y)\simeq U^*\otimes\Lambda^2W.
$$
For the sign-change action, $U$ is trivial and $W$ is the regular
representation of $(\ZZ/2)^3$.  Consequently,
$$
H^1(Y,\End T_Y)\simeq
\bigoplus_{\substack{1\neq \chi\in\widehat G}}\chi^{\oplus16},
$$
which has no invariant vectors.  Finite \'etale descent then proves
rigidity.  More generally, Corollary~\ref{cor:involutions} shows that a free
projective action of a finite group $G$ on a smooth intersection of four
quadrics satisfies
$$
h^1(Y/G,\End T_{Y/G})=\frac{16(7-t(G))}{|G|},
$$
where $t(G)$ is the number of nonidentity elements of order two.  Thus the
vanishing has an explicit representation-theoretic explanation.

\subsection{An obstructed tangent bundle}
Our second result concerns a different classical quotient.  Here the
tangent bundle is polystable rather than stable, and the calculation
determines the full semiuniversal base, not just its quadratic tangent
cone.

\begin{theorem}
\label{thm:igusa-kuranishi}
There is a Calabi--Yau threefold $Z$, obtained as a free $(\ZZ/2)^2$
quotient of a product of three elliptic curves, whose tangent bundle is a
direct sum of three pairwise nonisomorphic unitary flat line bundles and
satisfies
\[
\bigl(h^0(Z,\End T_Z),h^1(Z,\End T_Z),
      h^2(Z,\End T_Z),h^3(Z,\End T_Z)\bigr)=(3,6,6,3).
\]
Its semiuniversal analytic deformation base on the fixed threefold $Z$ is
isomorphic to the germ
\begin{equation}\label{eq:igusa-kuranishi}
\mathcal K=\bigl(V(ab,bc,ca,de,ef,fd)\subset\CC^6_{a,b,c,d,e,f},0\bigr).
\end{equation}
This germ is reduced and has nine irreducible components, each a
coordinate plane of dimension two.  In particular, it is singular at the
origin, and some first-order deformation of $T_Z$ has nonzero Yoneda
square in $H^2(Z,\End T_Z)$.
\end{theorem}

The underlying quotient is Igusa's example
\cite{Igusa1954}, in the form described by Oguiso--Sakurai
\cite[Example~2.17]{OguisoSakurai2001}.  The flat covering identifies its
endomorphism cohomology with invariant constant matrix-valued forms.
These representatives are closed under products, so the Kuranishi map is
exactly quadratic.  Calculating this product gives
\eqref{eq:igusa-kuranishi} and an explicit obstructed class.  Since
$T_Z$ is not simple, $\Def_Z(T_Z)$ is a semiuniversal deformation base;
it is not the coarse moduli space obtained by further identifying bundles
under automorphisms or $S$-equivalence.

\subsection{Relation to earlier work and organization}
Huybrechts proved unobstructedness of tangent bundles of Calabi--Yau
complete intersections in ordinary projective space
\cite[Corollary~2.2]{Huybrechts1995}.  For smooth hypersurfaces
$V\subset\PP^n$ with $n\ge4$ and degree $d\ge2$, Choe--Chung--Hwang
\cite[Theorems~1 and~3]{ChoeChungHwang2025} prove unobstructedness and the
formula
$
h^1(V,\End T_V)=\binom{n+d-1}{d}(d-1).
$
These results concern the covering complete intersections themselves;
rigidity of our quotient comes from taking invariants in a nonzero
cohomology group.  For obstructions, Aspinwall's noncompact calculation
\cite[Section~7]{Aspinwall2014} gives a nonzero Yoneda product and a cubic
potential.  Theorem~\ref{thm:igusa-kuranishi} instead computes the entire
fixed-base deformation germ on a compact threefold.  The comparison is
made more explicitly in Remark~\ref{rem:aspinwall}.

Related phenomena occur in complex dimension four. Gavran proves
rigidity of the tangent bundle for manifolds of $\mathrm{K3}^{[2]}$ type
\cite{Gavran2021}, while Bottini constructs a nonzero Yoneda square for
the tangent bundle of a generalized Kummer fourfold \cite{Bottini2026}.

Section~\ref{sec:prelim} recalls the deformation theory and finite \'etale
descent used throughout.  Section~\ref{sec:rigid} proves
Theorem~\ref{thm:rigid}, and Section~\ref{sec:characters} develops the
character formulas for free complete-intersection quotients.
Section~\ref{sec:igusa-obstructions} proves
Theorem~\ref{thm:igusa-kuranishi}.  Section~\ref{sec:infinite} shows that
an infinite fundamental group rules out tangent-bundle rigidity under
Convention~\ref{conv:cy} and records the scope of the two constructions.

\medskip
\noindent\textbf{Acknowledgments.}
Xueyuan Wan was supported by the National Key R\&D Program of China (Grant No.~2024YFA1013200) and the National Natural Science Foundation of China (Grant No.~12671100). The author used ChatGPT 6 Pro as an auxiliary tool in this work.
The author verified and completed all mathematical arguments and
takes full responsibility for the content of this paper.\vspace{3mm}

\section{Deformation-theoretic preliminaries}\label{sec:prelim}

Let $E$ be a holomorphic vector bundle on a compact complex manifold
$X$. We keep $X$ and the underlying smooth bundle fixed. A nearby
holomorphic structure has the form
\[
D_\alpha=\bar\partial_E+\alpha,
\qquad \alpha\in A^{0,1}(X,\End E),
\]
where $A^{0,q}(X,\End E)$ denotes the space of smooth
endomorphism-valued $(0,q)$-forms. The product on these forms combines
exterior multiplication and composition:
\[
(\omega\otimes A)\wedge(\theta\otimes B)
=(\omega\wedge\theta)\otimes(A\circ B).
\]
For forms $a$ and $b$ of degrees $p$ and $q$, put
$[a,b]=a\wedge b-(-1)^{pq}b\wedge a$. Together with the induced
Dolbeault operator, this bracket makes
\[
\mathfrak g_E=A^{0,\bullet}(X,\End E),
\qquad d=\bar\partial_E,
\]
a differential graded Lie algebra; see
\cite[Example~V.22]{Manetti2004}. Integrability of $D_\alpha$ is equivalent
to the Maurer--Cartan equation
\begin{equation}\label{eq:MC}
\bar\partial_E\alpha+\frac12[\alpha,\alpha]=0.
\end{equation}
Since $\alpha$ has degree one, $\tfrac12[\alpha,\alpha]=\alpha\wedge\alpha$.

Over $\mathbb C[t]/(t^2)$, write $\alpha=t\alpha_1$. The equation becomes
$\bar\partial_E\alpha_1=0$, and the change of trivialization
$g_t=\operatorname{id}_E+tu$ replaces $\alpha_1$ by
$\alpha_1+\bar\partial_Eu$. First-order deformations, with the central
fiber identified with $E$, are therefore parametrized by
\[
H^1(X,\End E)\simeq\Ext^1_X(E,E).
\]
Here local freeness of $E$ gives the last identification.

To extend $\xi=[\alpha_1]$ to second order, one must find $\alpha_2$
such that $\alpha=t\alpha_1+t^2\alpha_2$ satisfies
\eqref{eq:MC} modulo $t^3$. This requires
\[
\bar\partial_E\alpha_2+\alpha_1\wedge\alpha_1=0.
\]
Thus the primary obstruction is the Yoneda square
\begin{equation}\label{eq:primary-obstruction}
\operatorname{ob}_2(\xi)
=[\alpha_1\wedge\alpha_1]
=\xi\cup\xi\in H^2(X,\End E).
\end{equation}
Its vanishing is equivalent to the existence of a second-order extension;
higher-order obstructions may remain.

The analytic Kuranishi construction gives a holomorphic map
\[
\begin{gathered}
\kappa:U\subset H^1(X,\End E)\longrightarrow H^2(X,\End E),\\
\kappa(u)=u\cup u+\text{terms of degree at least three},
\end{gathered}
\]
whose zero-locus germ $(K_E,0)=(\kappa^{-1}(0),0)$ is a semiuniversal
deformation base. In particular, $T_0K_E\simeq H^1(X,\End E)$;
see \cite[Appendix~A]{Huybrechts1995} and
\cite[Chapter~V]{Manetti2004}. A nonzero class in $H^2(X,\End E)$ need not
be an obstruction, but a nonzero Yoneda square proves that $K_E$ is
singular. Throughout, we consider this semiuniversal base before any
further quotient by $\Aut(E)$.

\begin{lemma}\label{lem:CYduality}
If $X$ satisfies Convention~\ref{conv:cy}, then for every vector bundle
$E$,
\[
H^i(X,\End E)\simeq H^{3-i}(X,\End E)^*,
\quad \chi(X,\End E)=0.
\]
In particular, $h^1(X,\End E)=h^2(X,\End E)$.
\end{lemma}
\begin{proof}
The nondegenerate trace pairing $(A,B)\mapsto\operatorname{tr}(AB)$
identifies $\End E$ with its dual. Serre duality and
$K_X\simeq\OO_X$ then give
\[
H^i(X,\End E)
\simeq H^{3-i}\bigl(X,(\End E)^\vee\otimes K_X\bigr)^*
\simeq H^{3-i}(X,\End E)^*.
\]
The Euler characteristic vanishes by cancellation of the paired terms.
\end{proof}

The Euler characteristic therefore cannot detect rigidity; one must
compute $H^1(X,\End E)$ itself.

There is also a useful distinction between scalar and genuine
deformation directions.  The trace map gives a natural splitting
\[
\End E
=
\OO_X\cdot\operatorname{id}_E
\oplus
\End_0E,
\]
where $\End_0E$ denotes the sheaf of trace-free endomorphisms.  Hence
\[
H^1(X,\End E)
\simeq
H^1(X,\OO_X)
\oplus
H^1(X,\End_0E).
\]
The first summand corresponds to infinitesimal deformations obtained by
tensoring $E$ with a first-order deformation of the trivial line bundle,
whereas the second records the genuine deformation directions.  Under
Convention~\ref{conv:cy} we have $H^1(X,\OO_X)=0$, and therefore
\[
H^1(X,\End T_X)
\simeq
H^1(X,\End_0T_X).
\]
Thus, for the Calabi--Yau threefolds considered here, every nontrivial
first-order deformation of the tangent bundle would automatically be
genuine in the sense of Peternell.

\begin{lemma}
\label{lem:descent}
Let $p:Y\to X$ be a finite \'etale Galois cover of smooth
complex varieties, with covering group $G$, and let $E$ be
a vector bundle on $X$. Pullback induces natural isomorphisms
\[
p^*:H^i(X,\End E)
\xrightarrow{\sim}
H^i\bigl(Y,\End(p^*E)\bigr)^G,
\qquad i\geq 0,
\]
where the superscript $G$ denotes the invariant subspace
for the natural action of the covering group.
These isomorphisms preserve Yoneda products.

More generally, for every coherent sheaf $F$ on $X$,
pullback induces natural isomorphisms
\[
p^*:H^i(X,F)\xrightarrow{\sim}H^i(Y,p^*F)^G.
\]
Moreover, the differential of $p$ identifies
$T_Y$ with $p^*T_X$, and hence
\[
H^i(X,\End T_X)
\simeq H^i(Y,\End T_Y)^G.
\]
\end{lemma}

\begin{proof}
For a coherent sheaf $F$ on $X$, the Cartan--Leray
spectral sequence gives
\[
E_2^{a,b}
=
H^a\bigl(G,H^b(Y,p^*F)\bigr)
\Longrightarrow H^{a+b}(X,F);
\]
see \cite[\S5.2, Proposition~5.2.3 and Remark~3,
pp.~202--203]{Grothendieck1957}.
Since $G$ is finite and the coefficients are complex
vector spaces, averaging over $G$ makes the invariants
functor exact. Consequently,
$
H^a\bigl(G,H^b(Y,p^*F)\bigr)=0
$ for $a>0$.
The spectral sequence therefore degenerates, and its edge
maps, which are induced by pullback, yield
\[
H^i(X,F)\xrightarrow{\sim}H^i(Y,p^*F)^G.
\]

Taking $F=\End E$ and using the natural identification
$
p^*(\End E)\simeq\End(p^*E)
$
proves the assertion for endomorphism cohomology.
Since $p$ is \'etale, it is flat, so pullback is exact
and commutes with the splicing of extensions defining
Yoneda products. Thus, under the usual identifications
of endomorphism cohomology with self-$\Ext$ groups,
$
p^*(\xi\cup\eta)=p^*\xi\cup p^*\eta.
$

Finally, the \'etale condition implies that
$
dp:T_Y\xrightarrow{\sim}p^*T_X
$
is an isomorphism. It is $G$-equivariant because
$p\circ g=p$ for every $g\in G$, which proves the
tangent-bundle assertion.
\end{proof}

For a finite quotient, Lemma~\ref{lem:descent} allows rigidity to be
checked by taking invariants: a large space $H^1(Y,\End T_Y)$ may have
no invariant vectors. When $Y$ is an abelian variety, the same lemma
lets us compute obstruction products using invariant constant
matrix-valued forms on the cover.

\section{A stable rigid tangent bundle from four quadrics}\label{sec:rigid}

We first construct the free quotient, then compute the representation on
the tangent-bundle cohomology of its cover.

\subsection{An explicit smooth free quotient}

Index the homogeneous coordinates of $\PP^7$ by $v\in\FF_2^3$ and write
them as $[x_v]_{v\in\FF_2^3}$. Identifying $\FF_2$ with $\{0,1\}$, set
\[
t_v=v_1+2v_2+4v_3\in\{0,1,\ldots,7\}.
\]
Over $\CC$, define four diagonal quadrics
\[
q_a=\sum_{v\in\FF_2^3}t_v^a x_v^2,
\qquad a=0,1,2,3,
\]
with $t_v^0=1$, including when $t_v=0$, and let
\begin{equation}\label{eq:quadrics}
Y=V(q_0,q_1,q_2,q_3)\subset\PP^7
\end{equation}
be their common zero locus.

If the coordinates are ordered so that $t_v=0,1,\ldots,7$, the four
equations can be written compactly as
$
(q_0,q_1,q_2,q_3)^\top
=
C
(
x_0^2, x_1^2,\cdots, x_7^2
)^\top,
$
where
\begin{equation}\label{eq:vandermonde}
C=
\begin{pmatrix}
1&1&1&1&1&1&1&1\\
0&1&2&3&4&5&6&7\\
0&1&4&9&16&25&36&49\\
0&1&8&27&64&125&216&343
\end{pmatrix}.
\end{equation}
Every four-column minor of $C$ is a Vandermonde determinant
\[
\prod_{1\le r<s\le4}(i_s-i_r)\ne0,
\qquad 0\le i_1<i_2<i_3<i_4\le7.
\]
Thus any set of at most four columns is linearly independent. This will
prove both smoothness and freeness.

We write
\[
\langle g,v\rangle
=
g_1v_1+g_2v_2+g_3v_3
\pmod 2
\]
for the standard bilinear pairing on $\FF_2^3$.

\begin{proposition}\label{prop:geometry}
The variety $Y$ in \eqref{eq:quadrics} is a smooth connected simply
connected Calabi--Yau threefold.  The group $G=\FF_2^3$, acting by
\begin{equation}\label{eq:sign-action}
g\cdot x_v=(-1)^{\langle g,v\rangle}x_v,
\end{equation}
acts freely and preserves a nowhere-vanishing holomorphic three-form.
Consequently $X=Y/G$ satisfies Convention~\ref{conv:cy} and
$\pi_1(X)\simeq G$.
\end{proposition}
\begin{proof}
Consider the coordinatewise squaring map
\[
s:\PP^7\longrightarrow\PP^7,
\qquad
[x_v]_{v\in\FF_2^3}\longmapsto[x_v^2]_{v\in\FF_2^3}.
\]
The map is finite and surjective, and
$Y=s^{-1}(\PP(\ker C))$, where $\PP(\ker C)\simeq\PP^3$.
Thus $Y$ is nonempty. Every irreducible component has dimension at most
three, since $s$ is finite and its image lies in $\PP(\ker C)$.
On the other hand, four equations in $\PP^7$ give dimension at least
three for every component. Hence $Y$ is pure of dimension three.

Every point of $Y$ has at least five nonzero coordinates. Otherwise,
$C(x_v^2)_v=0$ would give a nontrivial relation among at most four
columns of $C$, contrary to the Vandermonde property.

Finally, the Jacobian matrix of the four quadrics is
\[
J(x)
=
\left(\frac{\partial q_a}{\partial x_v}\right)_{a,v}
=
2C\,\operatorname{diag}(x_v).
\]
Choosing four nonzero coordinates gives a nonzero $4\times4$ minor
of $J(x)$. The Jacobian criterion therefore shows that $Y$ is a smooth
complete intersection of dimension three.

Put $P=\PP^7_{\CC}$, and let $\iota:Y\hookrightarrow P$
denote the inclusion. Since $Y$ is a smooth complete
intersection of four quadrics, its normal bundle is
$
N_{Y/P}\simeq\OO_Y(2)^{\oplus4}.
$
The adjunction formula
\cite[Proposition~II.8.20]{Hartshorne1977} therefore gives
\[
\begin{aligned}
K_Y
\simeq K_P|_Y\otimes\det N_{Y/P}\simeq \OO_Y(-8)\otimes\OO_Y(8)
\simeq\OO_Y.
\end{aligned}
\]

We next compute the cohomology of $\OO_Y$.
The four defining quadrics form a regular sequence, so
their Koszul complex is a resolution of $\iota_*\OO_Y$;
see \cite[Tag~062F]{Stacks}. Explicitly, it is
\[
\begin{aligned}
0\longrightarrow\OO_P(-8)
&\longrightarrow\OO_P(-6)^{\oplus4}
\longrightarrow\OO_P(-4)^{\oplus6}\\
&\longrightarrow\OO_P(-2)^{\oplus4}
\longrightarrow\OO_P
\longrightarrow\iota_*\OO_Y
\longrightarrow0.
\end{aligned}
\]
The standard cohomology formula for line bundles on
projective space \cite[Tag~01XT]{Stacks} shows that
$\OO_P(-2)$, $\OO_P(-4)$, and $\OO_P(-6)$ have no
cohomology in any degree, while
\[
H^j(P,\OO_P(-8))=0
\qquad\text{for }j<7.
\]
Writing $\mathcal I_Y$ for the ideal sheaf of $Y$,
the truncated Koszul resolution consequently gives
\[
H^j(P,\mathcal I_Y)
\simeq H^{j+3}(P,\OO_P(-8))=0,
\qquad 0\leq j\leq3.
\]
Applying cohomology to
$
0\to\mathcal I_Y
\to\OO_P
\to\iota_*\OO_Y
\to 0
$
now yields
\[
H^0(Y,\OO_Y)=\CC,
\qquad
H^1(Y,\OO_Y)=H^2(Y,\OO_Y)=0.
\]
In particular, $Y$ is connected.

To apply the Lefschetz hyperplane theorem, set
\[
Y_0=P,\qquad
Y_r=V(q_0,\ldots,q_{r-1})
\quad (1\leq r\leq4).
\]
The same Vandermonde and Jacobian argument, applied to
the first $r$ rows of $C$, shows that $Y_r$ is smooth
of dimension $7-r$. Each inclusion
$Y_r\subset Y_{r-1}$ is the zero locus of a transverse
section of the positive line bundle
$\OO_{Y_{r-1}}(2)$.
The Lefschetz theorem, applied successively to these
inclusions, gives
\[
\pi_1(Y)\simeq\pi_1(P)=\{1\},
\qquad
H^2(P,\ZZ)\xrightarrow{\sim}H^2(Y,\ZZ);
\]
see \cite[Theorem~I and its corollary, p.~212]{Bott1959}.
Thus $Y$ is simply connected, and
\[
H^2(Y,\ZZ)
=
\ZZ\cdot c_1(\OO_Y(1)).
\]

Finally, we use the analytic exponential sequence.
Since $Y$ is projective, GAGA identifies its algebraic
and analytic coherent cohomology and Picard groups
\cite{Serre1956GAGA}. The exponential sequence gives
the exact segment
\[
H^1(Y,\OO_Y)
\longrightarrow \operatorname{Pic}(Y)
\xrightarrow{c_1} H^2(Y,\ZZ)
\longrightarrow H^2(Y,\OO_Y).
\]
The two outer groups vanish, so $c_1$ is an isomorphism.
Consequently,
$
\operatorname{Pic}(Y)
=
\ZZ\cdot[\OO_Y(1)].
$

For $g\in G=\FF_2^3$, let
$
D_g=\operatorname{diag}
\bigl((-1)^{\langle g,v\rangle}\bigr)_{v\in\FF_2^3}
$
be the corresponding linear transformation of $\CC^8$.
Since the action only changes the signs of the coordinates,
each defining quadric is invariant:
\[
q_a(D_gx)
=
\sum_v t_v^a
\bigl((-1)^{\langle g,v\rangle}x_v\bigr)^2
=
q_a(x),
\qquad a=0,1,2,3.
\]
Thus the action on $\PP^7$ preserves $Y$.

For $g\ne0$, the linear functional $v\mapsto\langle g,v\rangle$
takes each value four times. Thus $D_g$ has two four-dimensional
eigenspaces, and its projective fixed locus is the union of the
corresponding coordinate $\PP^3$'s. Every point in this locus has at most
four nonzero coordinates, so it is disjoint from $Y$. This proves
freeness. The action also remains faithful after projectivization,
since no $D_g$ with $g\ne0$ is scalar.

We also verify that $G$ preserves a nowhere-vanishing
holomorphic three-form on $Y$. Relabel the homogeneous
coordinates temporarily as $x_0,\ldots,x_7$, and set
$
\Omega
=
\sum_{j=0}^7(-1)^j x_j\,
dx_0\wedge\cdots\wedge\widehat{dx_j}
\wedge\cdots\wedge dx_7.
$
The adjunction trivialization of $K_Y$ is represented
by the complete-intersection residue
\[
\omega_Y
=
\operatorname{Res}_{q_0=\cdots=q_3=0}
\left(\frac{\Omega}{q_0q_1q_2q_3}\right).
\]
Here the numerator and denominator both have homogeneous
weight eight, so the quotient is a well-defined
meromorphic seven-form on $\PP^7$. Since the four
quadrics meet transversely, its residue is a
nowhere-vanishing holomorphic three-form on $Y$.

Under the coordinate action one has
$
g^*\Omega=(\det D_g)\Omega.
$
For $g\ne0$, the four negative eigenvalues give
$
\det D_g=(-1)^4=1,
$
and the same determinant identity holds for the identity
element. Since every $q_a$ is invariant, the residue
formula therefore gives
$
g^*\omega_Y=\omega_Y
\text{ for all }g\in G.
$

Let $X=Y/G$ and let $p:Y\to X$ be the quotient map.
The free finite action gives a smooth quotient and a
finite \'etale Galois cover of degree $|G|=8$;
see \cite[Tag~07S7]{Stacks}.
The quotient is projective: the natural $G$-linearization
of $\OO_Y(1)$ descends it to a line bundle $L$ on the
proper quotient $X$, with $p^*L\simeq\OO_Y(1)$.
The line bundle $L$ is ample because ampleness can be
checked after finite surjective pullback
\cite[Tag~0B5V]{Stacks}.

The invariant form $\omega_Y$ descends to a holomorphic
three-form $\omega_X$ satisfying
$
p^*\omega_X=\omega_Y.
$
Since $p$ is \'etale and $\omega_Y$ is nowhere vanishing,
$\omega_X$ is also nowhere vanishing. Thus
$
K_X\simeq\OO_X.
$
Furthermore, Lemma~\ref{lem:descent} gives
\[
H^1(X,\OO_X)
\simeq H^1(Y,\OO_Y)^G
=0.
\]
Finally, $Y$ is connected and simply connected, so
$p:Y\to X$ is the universal covering map. Its deck
transformation group is $G$, and therefore
$
\pi_1(X)\simeq G\simeq(\ZZ/2\ZZ)^3.
$
\end{proof}

The same construction works for any $4\times8$ coefficient matrix
whose four-column minors are nonzero. The Vandermonde matrix gives
an explicit member of this open set.

\subsection{A cohomology lemma for complete intersections}
We use the Euler, conormal, and normal exact sequences, following
\cite[\S\S2--4]{DGKM1989} and \cite{EastwoodHubsch1990}.\footnote{The
splitting on p.~119 of \cite{DGKM1989} shows that their notation
$\End T$ denotes the trace-free part of $T\otimes T^*$.
Its $H^1$ agrees with that of the full endomorphism bundle used here,
since $H^1(Y,\OO_Y)=0$; see also the note after their equation~(2.13).}
As \cite[\S4]{DGKM1989} explains, polynomial deformation counts need
not capture the full cohomology group. We compute that group directly
and retain the natural group actions on the exact sequences.
Let $Y\subset\PP^N$ be a smooth
Calabi--Yau complete intersection of dimension three, defined by a
regular sequence of degrees $d_1,\ldots,d_r\geq2$.  Put
\[
\begin{aligned}
W&=H^0(Y,\OO_Y(1)),& R_k&=H^0(Y,\OO_Y(k)),\\
N_Y&=N_{Y/\PP^N},& F&=T_{\PP^N}|_Y\otimes\Omega_Y^1.
\end{aligned}
\]
For $k<0$ we set $R_k=0$. Then
\begin{equation}\label{eq:ACM}
H^1(Y,\OO_Y(k))=H^2(Y,\OO_Y(k))=0
\quad\text{for every }k\in\ZZ.
\end{equation}
In fact, positive $k$ also follows from Kodaira vanishing,
negative $k$ from Serre duality, and $k=0$ from the complete-intersection
calculation above.  Restriction identifies $R_k$ with the degree-$k$
part of the homogeneous coordinate ring.  In particular,
$W\otimes R_{k-1}\to R_k$ is surjective for $k\geq1$.

For hypersurfaces, the following cohomology calculation is
\cite[Lemma~9]{ChoeChungHwang2025}. We give the complete-intersection
version needed here.

\begin{lemma}\label{lem:ambient}
In this situation,
\[
H^0(Y,F)=\CC,\qquad H^1(Y,F)=0.
\]
The generator of $H^0(Y,F)$ is the tangent inclusion
$T_Y\hookrightarrow T_{\PP^N}|_Y$.
\end{lemma}
\begin{proof}
Write $P=\PP^N$ and
$
E=\Omega_P^1(1)|_Y,
$ $
F=\Omega_Y^1\otimes T_P|_Y
  =\mathcal{H}om(T_Y,T_P|_Y).
$
We first compute the two lowest cohomology groups of
$\Omega_Y^1(1)$.

The cotangent Euler sequence, twisted by $\OO_P(1)$ and
restricted to $Y$, is
\begin{equation}\label{eq:euler-one}
0\longrightarrow E
\longrightarrow W\otimes\OO_Y
\longrightarrow\OO_Y(1)
\longrightarrow0;
\end{equation}
see \cite[II, Theorem~8.13]{Hartshorne1977}.
The right-hand map sends $w\otimes f$ to $f(w|_Y)$.
The preceding Koszul calculation gives
$H^0(Y,\OO_Y)=\mathbb C$
and shows that the restriction map
$
W=H^0(P,\OO_P(1))
\to H^0(Y,\OO_Y(1))
$
is an isomorphism: it is surjective, and there are no
linear relations because $d_j\geq2$.
Under these identifications, the map on global sections
in \eqref{eq:euler-one} is the identity on $W$.
Since $H^1(Y,\OO_Y)=0$ by \eqref{eq:ACM}, its cohomology
sequence gives
\[
H^0(Y,E)=H^1(Y,E)=0.
\]

Next, the conormal sequence, tensored with $\OO_Y(1)$, is
\begin{equation}\label{eq:conormal-one}
0\longrightarrow N_Y^*(1)
\longrightarrow E
\longrightarrow\Omega_Y^1(1)
\longrightarrow0;
\end{equation}
see \cite[II, Theorem~8.17]{Hartshorne1977}.
Because $Y$ is a complete intersection,
$
N_Y^*(1)\simeq\bigoplus_{j=1}^r\OO_Y(1-d_j).
$
Equation~\eqref{eq:ACM} therefore implies
\[
H^1(Y,N_Y^*(1))=H^2(Y,N_Y^*(1))=0.
\]
Combining these vanishings with those of $E$ in the
cohomology sequence of \eqref{eq:conormal-one}, we obtain
\begin{equation}\label{eq:omega-one-vanish}
H^0(Y,\Omega_Y^1(1))
=
H^1(Y,\Omega_Y^1(1))
=0.
\end{equation}

We now compute $H^0(Y,F)$.
Dualizing the cotangent Euler sequence gives the tangent
Euler sequence. Restricting it to $Y$ and tensoring with
$\Omega_Y^1$, we obtain
\begin{equation}\label{eq:euler-F}
0\longrightarrow\Omega_Y^1
\longrightarrow W^*\otimes\Omega_Y^1(1)
\longrightarrow F
\longrightarrow0.
\end{equation}
The weak Lefschetz theorem for smooth complete intersections
\cite[Proposition~5.4]{Movasati2021} and Hodge decomposition give
$
h^1(Y,\Omega_Y^1)=h^{1,1}(Y)=1.
$
Using \eqref{eq:omega-one-vanish} in the cohomology sequence
of \eqref{eq:euler-F}, we get an isomorphism
\[
H^0(Y,F)
\simeq H^1(Y,\Omega_Y^1)
\simeq\mathbb C.
\]
The tangent inclusion
$\iota_T:T_Y\hookrightarrow T_P|_Y$
is a nonzero section of
$F=\mathcal{H}om(T_Y,T_P|_Y)$.
It therefore spans $H^0(Y,F)$.

To prove that $H^1(Y,F)=0$, the next part of the same
cohomology sequence gives
\[
0\longrightarrow H^1(Y,F)
\longrightarrow H^2(Y,\Omega_Y^1)
\xrightarrow{\alpha}
W^*\otimes H^2(Y,\Omega_Y^1(1)).
\]
Since $\dim Y=3$ and $K_Y\simeq\OO_Y$, Serre duality
identifies the dual of $\alpha$ with the multiplication map
\begin{equation}\label{eq:multiply-tangent}
\mu:
W\otimes H^1(Y,T_Y(-1))
\longrightarrow H^1(Y,T_Y).
\end{equation}
In coordinates, this map sends
$\sum_i x_i\otimes\beta_i$ to $\sum_i x_i\beta_i$.
Consequently,
\[
H^1(Y,F)^*\simeq\operatorname{coker}(\mu).
\]
It remains to show that $\mu$ is surjective.

For $t=-1,0$, the restricted tangent Euler sequence is
\[
0\longrightarrow\OO_Y(t)
\longrightarrow W^*\otimes\OO_Y(t+1)
\longrightarrow T_P(t)|_Y
\longrightarrow0.
\]
By \eqref{eq:ACM}, both $H^1(Y,\OO_Y(t+1))$ and
$H^2(Y,\OO_Y(t))$ vanish. Hence
\[
H^1(Y,T_P(-1)|_Y)=H^1(Y,T_P|_Y)=0.
\]
The normal sequences
\[
0\longrightarrow T_Y(t)
\longrightarrow T_P(t)|_Y
\longrightarrow N_Y(t)
\longrightarrow0
\qquad(t=-1,0)
\]
therefore give surjective connecting maps
\[
\delta_t:
H^0(Y,N_Y(t))
\twoheadrightarrow H^1(Y,T_Y(t)).
\]
These maps commute with multiplication by linear forms,
by naturality of the cohomology long exact sequence.
Thus we have a commutative diagram
\[
\begin{tikzcd}[column sep=large,row sep=large]
W\otimes H^0(Y,N_Y(-1))
  \arrow[r]
  \arrow[d,two heads,"\mathrm{id}_W\otimes\delta_{-1}"']
&
H^0(Y,N_Y)
  \arrow[d,two heads,"\delta_0"]
\\
W\otimes H^1(Y,T_Y(-1))
  \arrow[r,"\mu"']
&
H^1(Y,T_Y).
\end{tikzcd}
\]
Since $N_Y\simeq\bigoplus_j\OO_Y(d_j)$, the top map is
the direct sum of the multiplication maps
\[
W\otimes R_{d_j-1}\longrightarrow R_{d_j},
\qquad
R_k=H^0(Y,\OO_Y(k)).
\]
Each is surjective by the preceding coordinate-ring
calculation: every homogeneous polynomial of degree $d_j$
is a sum of linear coordinates times polynomials of
degree $d_j-1$.

Surjectivity of the top and right-hand maps now implies that $\mu$
is surjective. Hence $H^1(Y,F)^*=\operatorname{coker}(\mu)=0$.
\end{proof}

The next identification comes from the normal-sequence description of
tangent-bundle deformations; compare \cite[Lemma~2.1]{Huybrechts1995}
and, for hypersurfaces, \cite[proof of Theorem~1]{ChoeChungHwang2025}.
For the quintic threefold, \cite[equation~(4.6)]{DGKM1989} already
identifies $H^1(Y,\End T_Y)$ with $H^0(\PP^4,\Omega_{\PP^4}^1(5))$.
We record the natural equivariant isomorphism needed for the quotient.

\begin{proposition}\label{prop:normal-formula}
There is a natural isomorphism
\begin{equation}\label{eq:normal-formula}
H^1(Y,\End T_Y)\simeq H^0(Y,\Omega_Y^1\otimes N_Y).
\end{equation}
It is equivariant for every finite linear group preserving $Y$.
\end{proposition}
\begin{proof}
Tensor the normal sequence by $\Omega_Y^1$ to obtain
\[
0\longrightarrow\End T_Y\longrightarrow F\longrightarrow
\Omega_Y^1\otimes N_Y\longrightarrow0.
\]
The generator of $H^0(Y, F)$ maps to zero, since the composite of the
tangent inclusion and the normal projection is zero.  The connecting
map is therefore injective, and $H^1(Y, F)=0$ makes it surjective.
All maps used are natural, which proves equivariance.
\end{proof}

\subsection{The representation for four quadrics}
Return to a smooth intersection of four quadrics in $\PP^7$, and let
$U=\mathrm{span}_{\mathbb{C}}\{q_0,q_1,q_2,q_3\}\subset\Sym^2W$ be its four-dimensional space of equations.
The natural normal bundle is
$N_Y=\mathcal{H}om(U,\mathcal{O}_Y(2))=\OO_Y(2)\otimes U^*$, so
\eqref{eq:normal-formula} becomes
\[
H^1(Y,\End T_Y)\simeq U^*\otimes H^0(Y,\Omega_Y^1(2)).
\]
Set $E_2=\Omega_P^1(2)|_Y$.
The restricted Euler sequence gives
\[
0\longrightarrow H^0(Y,E_2)
\longrightarrow W\otimes W
\xrightarrow{\mu} R_2
\longrightarrow0,
\]
where $\mu$ is multiplication of linear forms, followed by
restriction to $Y$. This map is surjective, and
$
R_2=\operatorname{Sym}^2W/U.
$

To describe its kernel, let
\[
s:\operatorname{Sym}^2W\longrightarrow W\otimes W,
\quad
s(ab)=\frac12(a\otimes b+b\otimes a)
\]
be the natural symmetrization map. Multiplication kills the
alternating tensors, while a symmetric tensor restricts to
zero on $Y$ precisely when it represents a defining quadric.
Thus
\[
H^0(Y,E_2)=\ker\mu
=\Lambda^2W\oplus s(U).
\]

The twisted conormal sequence is
\[
0\longrightarrow U\otimes\mathcal O_Y
\xrightarrow{d}E_2
\longrightarrow\Omega_Y^1(2)
\longrightarrow0.
\]
Since $H^0(Y,\mathcal O_Y)=\mathbb C$ and
$H^1(Y,\mathcal O_Y)=0$, it yields
\[
0\longrightarrow U
\xrightarrow{d}H^0(Y,E_2)
\longrightarrow H^0(Y,\Omega_Y^1(2))
\longrightarrow0.
\]
Under the Euler identification, the differential of a quadric
$q$ corresponds to $2s(q)$. Hence the image of $d$ is exactly
$s(U)$, and therefore
\[
H^0(Y,\Omega_Y^1(2))\cong\Lambda^2W.
\]
We obtain the following formula.
\begin{proposition}\label{prop:wedge-formula}
Let $Y\subset\mathbb P^7$ be a smooth intersection of four
quadrics. With the natural equivariant structures,
\[
H^1(Y,\operatorname{End}T_Y)
\cong U^*\otimes\Lambda^2W.
\]
In particular,
\[
h^1(Y,\operatorname{End}T_Y)=4\binom82=112.
\]
\end{proposition}

Let $\widehat G=\operatorname{Hom}(G,\CC^*)$ be the character group.
For the action in \eqref{eq:sign-action}, each character
$\chi_v(g)=(-1)^{\langle g,v\rangle}$ of $G=(\ZZ/2)^3$ occurs once in
$W$, while $U$ and $U^*$ are four-dimensional trivial representations.
If $e_\chi$ and $e_\psi$ span the character summands indexed by $\chi$
and $\psi$, then $e_\chi\wedge e_\psi$ has character $\chi\psi$.
For $\chi\ne\psi$ this character is nontrivial, since
$\chi^{-1}=\chi$.

Write $\CC_\rho$ for the one-dimensional representation with character
$\rho$. A fixed nontrivial character $\rho$ occurs in $\Lambda^2W$
once for each unordered pair $\{\chi,\chi\rho\}$. There are four such
pairs, so
\[
\Lambda^2W\cong\bigoplus_{\rho\ne1}\CC_\rho^{\oplus4}.
\]
Tensoring with $U^*$ gives
\begin{equation}\label{eq:seven-characters}
H^1(Y,\operatorname{End}T_Y)
\cong U^*\otimes\Lambda^2W
\cong
\bigoplus_{\rho\neq1}\mathbb C_\rho^{\oplus16}.
\end{equation}

There is no trivial summand. Lemma~\ref{lem:descent} therefore gives
\[
H^1(X,\operatorname{End}T_X)
\cong H^1(Y,\operatorname{End}T_Y)^G=0.
\]

\subsection{Stability and the numerical invariants}
Let \(H=c_1(\mathcal O_Y(1))\) be the hyperplane class on \(Y\).
The restricted Euler sequence gives
$
c(T_{\mathbb P^7}|_Y)=(1+H)^8.
$
Since \(Y\) is a complete intersection of four quadrics, its
normal bundle is \(\mathcal O_Y(2)^{\oplus4}\). The normal
sequence therefore gives
\[
c(T_Y)
=\frac{c(T_{\mathbb P^7}|_Y)}
       {c(\mathcal O_Y(2)^{\oplus4})}
=\frac{(1+H)^8}{(1+2H)^4}
=1+4H^2-8H^3.
\]
Here terms of degree higher than \(H^3\) vanish because
\(\dim_{\mathbb C}Y=3\). In particular,
\(c_3(T_Y)=-8H^3\).

The degree of \(Y\) is the product of the degrees of its four
defining equations, so
$
\int_Y H^3=\deg Y=2^4=16.
$
Thus 
$
e(Y)=\int_Y c_3(T_Y)=-8\cdot16=-128.
$
Since \(G\) acts freely and has order \(8\), the quotient map
\(\pi:Y\to X\) is an eight-sheeted covering. Consequently,
$
e(X)=\frac{e(Y)}{8}=-16.
$

By weak Lefschetz, \(H^2(Y,\mathbb C)=\mathbb C H\).
Every element of \(G\) preserves the hyperplane class, so
the action on this cohomology group is trivial. Pullback
therefore identifies
$
H^2(X,\mathbb C)
\cong H^2(Y,\mathbb C)^G
=\mathbb C H.
$
Since this identification preserves Hodge types and \(H\)
has type \((1,1)\), we obtain \(h^{1,1}(X)=1\).
The Euler characteristic of a Calabi--Yau threefold satisfies
$
e(X)=2\bigl(h^{1,1}(X)-h^{2,1}(X)\bigr).
$
Substituting \(e(X)=-16\) and \(h^{1,1}(X)=1\), we find
$
h^{2,1}(X)=9.
$

Contraction with a nowhere-vanishing holomorphic three-form
identifies \(T_X\) with \(\Omega_X^2\), and hence
\[
h^1(X,T_X)=h^{2,1}(X)=9.
\]
The Bogomolov--Tian--Todorov theorem then gives a smooth
nine-dimensional local deformation space for \(X\). Thus the complex
structure varies, although its tangent bundle is rigid on the fixed
threefold.

We now prove that \(T_X\) is slope stable with respect to every
ample polarization. Let \(L\) be an arbitrary ample line bundle
on \(X\), and set \(A=\pi^*L\), where \(\pi:Y\to X\) is the
quotient map.

By Yau's theorem, \(Y\) admits a Ricci-flat K\"ahler metric
in the class \(c_1(A)\) \cite{Yau1978}.
Since \(Y\) is simply connected, has trivial canonical bundle,
and satisfies
$
H^1(Y,\mathcal O_Y)=H^2(Y,\mathcal O_Y)=0,
$
the Beauville--Bogomolov decomposition theorem implies that
its holonomy group is \(SU(3)\), acting irreducibly on the
tangent space \cite[Theorem~1]{Beauville1983}.

The induced Hermitian metric on \(T_Y\) is Hermite--Einstein.
Hence \(T_Y\) is polystable with respect to \(A\).
Any nontrivial decomposition into stable summands would give
a proper parallel subbundle, contradicting the irreducibility
of the holonomy representation. Thus \(T_Y\) is slope stable
with respect to \(A\)
\cite[Theorem~5.8.3]{Kobayashi1987}.

Suppose, for a contradiction, that \(T_X\) is not slope stable
with respect to \(L\). Then there is a coherent subsheaf
\(\mathcal F\subset T_X\) such that
$
0<\operatorname{rk}\mathcal F<3,
$ $
\mu_L(\mathcal F)\geq\mu_L(T_X).
$
Since \(\pi\) is finite \'etale, pullback gives an inclusion
$
\pi^*\mathcal F\hookrightarrow\pi^*T_X\cong T_Y.
$
Moreover, the projection formula and \(\deg\pi=8\) give
$
\mu_A(\pi^*\mathcal F)=8\mu_L(\mathcal F),
$ $
\mu_A(T_Y)=8\mu_L(T_X).
$
Consequently,
$
\mu_A(\pi^*\mathcal F)\geq\mu_A(T_Y),
$
contradicting the stability of \(T_Y\).
Therefore \(T_X\) is slope stable with respect to \(L\).
Since \(L\) was arbitrary, this holds for every ample
polarization on \(X\).

Together with Proposition~\ref{prop:geometry}, the vanishing
\(H^1(X,\operatorname{End}T_X)=0\) obtained from
\eqref{eq:seven-characters}, and the Hodge-number calculation
above, this completes the proof of Theorem~\ref{thm:rigid}.

Since every slope-stable vector bundle is simple
\cite[Corollary 5.7.14]{Kobayashi1987}, the stability of \(T_X\)
implies that
$
H^0(X,\operatorname{End}T_X)
=\mathbb C\,\operatorname{id}_{T_X}.
$
Since $\operatorname{Ext}_X^i(T_X,T_X)\cong H^i(X,\operatorname{End}T_X)$,
Lemma~\ref{lem:CYduality} gives the complete self-Ext table
\[
\operatorname{Ext}^i_X(T_X,T_X)=
\begin{cases}
\mathbb C,&i=0,3,\\
0,&\text{otherwise}.
\end{cases}
\]
Together with \(T_X\otimes K_X\cong T_X\), this shows that
\(T_X\) is a \(3\)-spherical bundle in the sense of
\cite[Definition~1.1(a)]{SeidelThomas2001}.

More generally, let \(H_0\subset G\) be any subgroup and write
\(Z=Y/H_0\). Since \(H_0\) acts freely, Lemma~\ref{lem:descent}
identifies
$
H^1(Z,\operatorname{End}T_Z)
\cong
H^1(Y,\operatorname{End}T_Y)^{H_0}.
$
By \eqref{eq:seven-characters}, each nontrivial character
\(\rho\in\widehat G\) occurs with multiplicity \(16\).
Its summands contribute to the \(H_0\)-invariant subspace
exactly when \(\rho|_{H_0}=1\).

The characters of \(G\) that are trivial on \(H_0\) are
precisely the pullbacks of the characters of \(G/H_0\).
There are \(8/|H_0|\) such characters, including the trivial
character of \(G\), which does not occur in
\eqref{eq:seven-characters}. Therefore
\[
h^1(Z,\operatorname{End}T_Z)
=
16\left(\frac{8}{|H_0|}-1\right).
\]
For subgroup orders \(1,2,4,8\), these dimensions are
\(112,48,16,0\), respectively.

\section{Character formulas for free complete-intersection quotients}
\label{sec:characters}

The equivariant calculation of Section~\ref{sec:rigid} extends to all
Calabi--Yau complete intersections in a single projective space.  The
result below expresses the dimension on a free quotient in terms of the
representations carried by its defining equations.
Let \(Y\subset\mathbb P^N\) be a smooth Calabi--Yau complete
intersection of dimension three, with all defining degrees
at least two. Let \(G\) be a finite group acting freely on
\(Y\) by projective transformations, and set \(X=Y/G\).

We first show that the action preserves a nowhere-vanishing
holomorphic three-form. Since
$
H^0(Y,\mathcal O_Y)\cong H^3(Y,\mathcal O_Y)\cong\mathbb C,
$ $
H^1(Y,\mathcal O_Y)=H^2(Y,\mathcal O_Y)=0,
$
the holomorphic Lefschetz number of any \(g\ne1\) is
$
L(g,\mathcal O_Y)
=
1-\operatorname{tr}\bigl(g\mid H^3(Y,\mathcal O_Y)\bigr).
$
The action is free, so \(g\) has no fixed points and this
Lefschetz number is zero. Since \(H^3(Y,\mathcal O_Y)\) is
one-dimensional, \(G\) acts trivially on it. By the naturality
of Serre duality, \(G\) also acts trivially on \(H^0(Y,K_Y)\).

Thus a nowhere-vanishing holomorphic three-form on \(Y\)
descends to \(X\), giving \(K_X\cong\mathcal O_X\).
Moreover,
$
H^i(X,\mathcal O_X)
\cong H^i(Y,\mathcal O_Y)^G=0,
$ $i=1,2.
$
As \(X\) is smooth and projective, it satisfies
Convention~\ref{conv:cy}.

To carry out the equivariant calculations, we replace the
projective action by a linear action of a finite central
extension. Let \(\Gamma\) be the inverse image of \(G\) under
$
\mathrm{SL}_{N+1}(\mathbb C)
\to
\mathrm{PGL}_{N+1}(\mathbb C).
$
Then
\[
1\longrightarrow\mu_{N+1}
\longrightarrow\Gamma
\longrightarrow G
\longrightarrow1,
\]
where
$
\mu_{N+1}=\{\zeta I:\zeta^{N+1}=1\}
$
is the scalar kernel. In particular, \(\Gamma\) is finite.
Its linear action gives a natural linearization of
\(\mathcal O_Y(1)\), even when this line bundle does not
admit a \(G\)-linearization. This use of finite central
extensions is standard; see \cite[\S3]{Hua2011} and
\cite[\S3.1 and \S4.1]{Braun2011}.

We therefore perform the subsequent equivariant calculations
using \(\Gamma\). The scalar kernel may act nontrivially on
intermediate spaces involving \(\mathcal O_Y(1)\). However,
it acts trivially on \(Y\), hence also on \(T_Y\),
\(\operatorname{End}T_Y\), and their cohomology.
The resulting representations on
\(H^i(Y,\operatorname{End}T_Y)\) therefore descend to \(G\).

Let \(S=\mathbb C[x_0,\ldots,x_N]\), let \(I\subset S\) be
the homogeneous ideal of \(Y\), and write
\(S_+=\bigoplus_{k>0}S_k\).
For each degree \(d\), define
$
U_d=(I/S_+I)_d,$ $m_d=\dim_{\mathbb C}U_d.$
Thus \(U_d\) records the degree-\(d\) defining equations
modulo those obtained by multiplying lower-degree equations
by homogeneous polynomials. Its dimension \(m_d\) is the
number of degree-\(d\) generators in a minimal homogeneous
generating set for \(I\).

The action of \(\Gamma\) preserves both \(I_d\) and
\((S_+I)_d\), so \(U_d\) is naturally a
\(\Gamma\)-representation. Since \(\Gamma\) is finite and
we work over \(\mathbb C\), we may choose a
\(\Gamma\)-equivariant splitting of
$
0\to(S_+I)_d
\to I_d
\to U_d
\to 0.
$
These splittings realize the spaces \(U_d\) as
\(\Gamma\)-invariant spaces of homogeneous polynomials
whose bases together minimally generate \(I\).

With these choices, the defining equations induce a
\(\Gamma\)-equivariant isomorphism
$
N_Y^*
=\mathcal I_Y/\mathcal I_Y^2
\cong
\bigoplus_d\mathcal O_Y(-d)\otimes U_d,
$
because \(Y\) is a complete intersection.
Taking duals gives the corresponding description of the
normal bundle:
\[
N_Y\cong
\bigoplus_d\mathcal O_Y(d)\otimes U_d^*.
\]
Write $R_k=H^0(Y,\OO_Y(k))$ for $k\ge0$ and $R_k=0$ for
$k<0$, and put $W=R_1$ and $b=h^1(Y,\End T_Y)$.

\begin{proposition}\label{prop:character-formula}
With the preceding notation,
\begin{equation}\label{eq:general-quotient-count}
h^1(X,\End T_X)=
\frac{b+\sum_dm_d^2}{|G|}-\sum_d\dim(\End U_d)^\Gamma.
\end{equation}
More precisely, for $1\ne g\in G$ and any lift $\gamma\in\Gamma$,
\begin{equation}\label{eq:trace-End-negative}
\operatorname{tr}(g\mid H^1(Y,\End T_Y))
=-\sum_d\left|\operatorname{tr}(\gamma\mid U_d)\right|^2.
\end{equation}
\end{proposition}
\begin{proof}
Write
$
V=H^1(Y,\operatorname{End}T_Y),
$ $
E_d=\Omega^1_{\mathbb P^N}(d)|_Y,
$
Proposition~\ref{prop:normal-formula} gives a
\(\Gamma\)-equivariant isomorphism
\[
V\cong
\bigoplus_d U_d^*\otimes H^0(Y,\Omega_Y^1(d)).
\]

The restricted Euler sequence, twisted by
\(\mathcal O_Y(d)\), is
\[
0\longrightarrow E_d
\longrightarrow W\otimes\mathcal O_Y(d-1)
\longrightarrow\mathcal O_Y(d)
\longrightarrow0.
\]
Since multiplication \(W\otimes R_{d-1}\to R_d\) is
surjective, its global sections give an exact sequence
\[
0\longrightarrow H^0(Y,E_d)
\longrightarrow W\otimes R_{d-1}
\longrightarrow R_d
\longrightarrow0.
\]
The twisted conormal sequence is
\[
0\longrightarrow
\bigoplus_e U_e\otimes\mathcal O_Y(d-e)
\longrightarrow E_d
\longrightarrow\Omega_Y^1(d)
\longrightarrow0.
\]
By \eqref{eq:ACM}, we have
\(H^1(Y,\mathcal O_Y(d-e))=0\). Thus this sequence gives
\[
0\longrightarrow
\bigoplus_e U_e\otimes R_{d-e}
\longrightarrow H^0(Y,E_d)
\longrightarrow H^0(Y,\Omega_Y^1(d))
\longrightarrow0.
\]
Combining these two exact sequences yields
\begin{equation}\label{eq:general-character-ring}
[V]
=
\sum_d[U_d^*]
\left(
[W\otimes R_{d-1}]-[R_d]
-\sum_e[U_e\otimes R_{d-e}]
\right)
\end{equation}
in the representation ring of \(\Gamma\).

Now let \(1\ne g\in G\), and choose a lift
\(\gamma\in\Gamma\). For \(k>0\), Kodaira vanishing and
\(K_Y\cong\mathcal O_Y\) give
$
H^i(Y,\mathcal O_Y(k))=0
$ $ (i>0).$
Since \(g\) has no fixed points, the holomorphic Lefschetz
formula therefore gives
\begin{equation}\label{eq:free-positive-trace}
0=
\sum_i(-1)^i
\operatorname{tr}\bigl(
\gamma\mid H^i(Y,\mathcal O_Y(k))
\bigr)
=
\operatorname{tr}(\gamma\mid R_k).
\end{equation}
Here we use the fixed-point-free case recalled in
\cite[Theorem~2.1]{Hua2011}.

Since every defining degree satisfies \(d\ge2\), the
terms involving \(R_{d-1}\) and \(R_d\) contribute zero
to the trace in \eqref{eq:general-character-ring}.
Among the terms involving \(R_{d-e}\), those with \(d>e\)
have zero trace by \eqref{eq:free-positive-trace}, and
those with \(d<e\) vanish because \(R_{d-e}=0\).
Thus only \(d=e\) remains. Since \(R_0=\mathbb C\) is
the trivial representation, we obtain
$$
\operatorname{tr}(g\mid V)
=
-\sum_d
\operatorname{tr}(\gamma\mid U_d^*)
\operatorname{tr}(\gamma\mid U_d).
$$
Finite-group representations over \(\mathbb C\) admit
invariant Hermitian inner products, so
$
\operatorname{tr}(\gamma\mid U_d^*)
=
\overline{\operatorname{tr}(\gamma\mid U_d)}.
$
Consequently,
\[
\operatorname{tr}(g\mid V)
=
-\sum_d
\left|\operatorname{tr}(\gamma\mid U_d)\right|^2,
\qquad g\ne1,
\]
which proves \eqref{eq:trace-End-negative}.

The scalar kernel of \(\Gamma\to G\) acts on each
\(U_d\) by scalars of absolute value one. Hence the
absolute value squared in this formula is independent
of the chosen lift. Moreover, the conjugation action
on \(\operatorname{End}U_d\) descends to \(G\), with
character
$
g\longmapsto
\left|\operatorname{tr}(\gamma_g\mid U_d)\right|^2.
$
Averaging this character gives
\[
\frac1{|G|}\sum_{g\in G}
\left|\operatorname{tr}(\gamma_g\mid U_d)\right|^2
=
a_d,
\qquad
a_d:=\dim(\operatorname{End}U_d)^\Gamma.
\]

Finally, let \(D=\dim V\) and choose \(\gamma_1=1\).
By Lemma~\ref{lem:descent} and the character averaging
formula,
\[
\begin{aligned}
h^1(X,\operatorname{End}T_X)
&=\dim V^G=\frac1{|G|}
\left(
D-\sum_d\sum_{g\ne1}
\left|\operatorname{tr}(\gamma_g\mid U_d)\right|^2
\right)\\
&=\frac1{|G|}
\left(D-\sum_d(|G|a_d-m_d^2)\right)\\
&=\frac{D+\sum_d m_d^2}{|G|}-\sum_d a_d.
\end{aligned}
\]
This is \eqref{eq:general-quotient-count}.
\end{proof}

\subsection{The five complete-intersection types}

We first list the possible multidegrees of $Y$.
Since $Y\subset\PP^N$ is a three-dimensional complete intersection
defined by $r$ equations, we have $r=N-3$.
The Calabi--Yau condition, together with adjunction, gives
$\sum_j d_j=N+1$. Hence
$
\sum_{j=1}^r(d_j-1)
=(N+1)-(N-3)=4.
$
Each $d_j-1$ is a positive integer, since $d_j\geq2$.
The five partitions
\[
4,\qquad 3+1,\qquad 2+2,\qquad
2+1+1,\qquad 1+1+1+1
\]
therefore give the five types
\[
(5),\qquad (2,4),\qquad (3,3),\qquad
(2,2,3),\qquad (2,2,2,2).
\]
Their ambient spaces are determined by $N=r+3$.

To compute the cohomology dimensions, write
\[
a_k=\dim R_k=h^0(Y,\OO_Y(k)),
\quad
m_d=\#\{j:d_j=d\}=\dim U_d,
\]
with the conventions $a_0=1$ and $a_k=0$ for $k<0$.
The numbers $a_k$ are the coefficients of the Hilbert series
\begin{equation}\label{eq:CICY-Hilbert}
\sum_{k\geq0}a_k t^k
=\frac{\prod_{j=1}^r(1-t^{d_j})}{(1-t)^{N+1}}.
\end{equation}
Indeed, the polynomial ring in $N+1$ variables has Hilbert
series $(1-t)^{-(N+1)}$. Since the defining equations form a
regular sequence, imposing an equation of degree $d_j$
multiplies the Hilbert series by $1-t^{d_j}$.

Taking dimensions in the Euler and conormal calculations
used in \eqref{eq:general-character-ring} gives
\begin{equation}\label{eq:CICY-cotangent-dimension}
h^0(Y,\Omega_Y^1(d))
=(N+1)a_{d-1}-a_d-\sum_e m_ea_{d-e}.
\end{equation}
Proposition~\ref{prop:normal-formula} then gives
\begin{equation}\label{eq:CICY-End-dimension}
b:=h^1(Y,\operatorname{End}T_Y)
=\sum_d m_d\,h^0(Y,\Omega_Y^1(d)).
\end{equation}
Table~\ref{tab:CICY} records the dimensions, together with the
numerical terms needed in \eqref{eq:general-quotient-count}.
The remaining term in that formula depends on the representations
carried by the defining equations.

\begin{table}[htbp]
\centering
\caption{Tangent-bundle cohomology for the five
Calabi--Yau complete-intersection types in ordinary
projective space.}
\label{tab:CICY}
\begin{tabular}{lcrrr}
\toprule
Type & Ambient space
& $b=h^1(Y,\operatorname{End}T_Y)$
& $\sum_d m_d^2$
& $b+\sum_d m_d^2$\\
\midrule
$(5)$       & $\PP^4$ & 224 & 1  & 225\\
$(2,4)$     & $\PP^5$ & 190 & 2  & 192\\
$(3,3)$     & $\PP^5$ & 140 & 4  & 144\\
$(2,2,3)$   & $\PP^6$ & 140 & 5  & 145\\
$(2,2,2,2)$ & $\PP^7$ & 112 & 16 & 128\\
\bottomrule
\end{tabular}
\end{table}

Here are the Hilbert coefficients and substitutions used in the table.

\emph{Type $(5)$ in $\PP^4$.}
There is one defining equation of degree five, so $m_5=1$.
The Hilbert series gives
\[
a_4=\binom84=70,
\qquad
a_5=\binom94-1=125.
\]
Hence
\[
h^0(Y,\Omega_Y^1(5))
=5\cdot70-125-1=224,
\qquad
b=224.
\]

\emph{Type $(2,4)$ in $\PP^5$.}
Here $m_2=m_4=1$, and
\[
a_1=6,\quad
a_2=21-1=20,\quad
a_3=56-6=50,\quad
a_4=126-21-1=104.
\]
Formula~\eqref{eq:CICY-cotangent-dimension} gives
\[
\begin{aligned}
h^0(Y,\Omega_Y^1(2))
&=6\cdot6-20-1=15,\\
h^0(Y,\Omega_Y^1(4))
&=6\cdot50-104-20-1=175.
\end{aligned}
\]
Since each degree occurs once,
\[
b=15+175=190.
\]

\emph{Type $(3,3)$ in $\PP^5$.}
There are two cubic equations, so $m_3=2$.
We have
\[
a_2=\binom75=21,
\qquad
a_3=\binom85-2=54.
\]
Thus
\[
h^0(Y,\Omega_Y^1(3))
=6\cdot21-54-2=70.
\]
This contribution occurs twice, giving
\[
b=2\cdot70=140.
\]

\emph{Type $(2,2,3)$ in $\PP^6$.}
Here $m_2=2$ and $m_3=1$. The required Hilbert coefficients are
\[
a_1=7,\qquad
a_2=28-2=26,\qquad
a_3=84-2\cdot7-1=69.
\]
Consequently,
\[
\begin{aligned}
h^0(Y,\Omega_Y^1(2))
&=7\cdot7-26-2=21,\\
h^0(Y,\Omega_Y^1(3))
&=7\cdot26-69-2\cdot7-1=98.
\end{aligned}
\]
The quadratic contribution occurs twice and the cubic
contribution once, so
\[
b=2\cdot21+98=140.
\]

\emph{Type $(2,2,2,2)$ in $\PP^7$.}
There are four quadratic equations, so $m_2=4$.
Since
\[
a_1=8,\qquad
a_2=\binom97-4=32,
\]
we obtain
\[
h^0(Y,\Omega_Y^1(2))
=8\cdot8-32-4=28.
\]
Therefore
\[
b=4\cdot28=112.
\]

The value $224$ for the quintic and $140$ for the intersection of two
cubics already appear in \cite[equation~(2.13) and \S3]{DGKM1989}.
Their \S3 also explains how the exact-sequence method extends to the
other complete intersections in a single projective space, although
the remaining three values in Table~\ref{tab:CICY} are not listed there.
The numerical dimensions are included here to supply the data for
the quotient formula. For further early calculations and deformation
results, see \cite{EastwoodHubsch1990,Huybrechts1995}.

\subsection{Four quadrics and involutions}

We use the Lefschetz character calculation for the spaces of coordinates
and equations, as in \cite[\S\S2--3]{Hua2011}, to compute the character
on $H^1(Y,\End T_Y)$.
For an intersection of four quadrics, the dimension of the
tangent bundle deformation space of a free quotient is determined
by two numerical invariants of the acting group: its order and
the number of its elements of order two.
We call a nonidentity element $g$ with $g^2=1$ an involution.

\begin{corollary}\label{cor:involutions}
Suppose that a finite group $G$ acts freely by projective
transformations on a smooth intersection
$Y\subset\PP^7$ of four quadrics, and set $X=Y/G$.
Let
$
t(G)=\#\{g\in G:g\ne1, g^2=1\}
$
be the number of involutions in $G$. Then
\begin{equation}\label{eq:involution-formula}
h^1(X,\End T_X)
=\frac{16(7-t(G))}{|G|}.
\end{equation}
In particular $t(G)\leq7$ and $T_X$ is infinitesimally
rigid if and only if $t(G)=7$.
\end{corollary}

\begin{proof}
Put
$
V=H^1(Y,\End T_Y).
$
By Proposition~\ref{prop:wedge-formula},
$$
V\simeq U^*\otimes\Lambda^2W,
$$
where $U$ is the four-dimensional space of defining quadrics
and $\dim W=8$. In particular,
$
\dim V=4\binom82=112.
$
By Lemma~\ref{lem:descent}, the desired dimension is
$\dim V^G$. We compute it by averaging the traces of the
elements of $G$ on $V$.

Let $g\ne1$, and choose a lift $\gamma\in\Gamma$.
The trace-vanishing statement
\eqref{eq:free-positive-trace}, applied in degrees one and
two, gives
$
\operatorname{tr}(\gamma\mid W)
=\operatorname{tr}(\gamma\mid R_2)=0.
$
The defining quadrics are the kernel of the restriction map
on quadratic polynomials, so there is an exact sequence
$
0\to U
\to\Sym^2W
\to R_2
\to 0.
$
Consequently,
\[
\operatorname{tr}(\gamma\mid U)
=\operatorname{tr}(\gamma\mid\Sym^2W).
\]

The standard trace identities for the symmetric and exterior
squares are
$
\operatorname{tr}(\gamma\mid\Sym^2W)
=\frac12(
\operatorname{tr}(\gamma\mid W)^2
+\operatorname{tr}(\gamma^2\mid W)
),$ $
\operatorname{tr}(\gamma\mid\Lambda^2W)
=\frac12(
\operatorname{tr}(\gamma\mid W)^2
-\operatorname{tr}(\gamma^2\mid W)
).
$
These follow by writing the eigenvalues of $\gamma$ on $W$
as $\lambda_1,\ldots,\lambda_8$ and summing
$\lambda_i\lambda_j$ over $i\leq j$ and $i<j$, respectively.
Since $\operatorname{tr}(\gamma\mid W)=0$, we obtain
\[
\operatorname{tr}(\gamma\mid U)
=\frac12\operatorname{tr}(\gamma^2\mid W),\quad
\operatorname{tr}(\gamma\mid\Lambda^2W)
=-\frac12\operatorname{tr}(\gamma^2\mid W).
\]

If $g^2\ne1$, then $\gamma^2$ is a lift of the nonidentity
element $g^2$. Applying
\eqref{eq:free-positive-trace} once more gives
$
\operatorname{tr}(\gamma^2\mid W)=0.
$
It follows that
\[
\operatorname{tr}(g\mid V)=0.
\]

If $g$ is an involution, then $\gamma^2$ lies in the scalar
kernel of $\Gamma\to G$. Thus it acts on $W$ as
$c\,\operatorname{id}_W$ for a root of unity $c$, and
\[
\operatorname{tr}(\gamma\mid U)=4c,
\qquad
\operatorname{tr}(\gamma\mid\Lambda^2W)=-4c.
\]
Because $\Gamma$ is finite, the character of the dual
representation is the complex conjugate character. Hence
\[
\operatorname{tr}(\gamma\mid U^*)
=4\overline c=4c^{-1}.
\]
Using $V\simeq U^*\otimes\Lambda^2W$, we conclude that
\[
\operatorname{tr}(g\mid V)
=(4c^{-1})(-4c)=-16.
\]
In particular, this value is independent of the chosen lift.

We have therefore shown that
\[
\operatorname{tr}(g\mid V)=
\begin{cases}
112,&g=1,\\
-16,&g\ne1\text{ and }g^2=1,\\
0,&g^2\ne1.
\end{cases}
\]
Averaging these traces gives
\[
\begin{aligned}
h^1(X,\End T_X)
&=\dim V^G=\frac1{|G|}
  \sum_{g\in G}\operatorname{tr}(g\mid V)\\
&=\frac{112-16t(G)}{|G|}
 =\frac{16(7-t(G))}{|G|}.
\end{aligned}
\]
Since this dimension is nonnegative, $t(G)\leq7$.
It vanishes exactly when $t(G)=7$.
\end{proof}

For $G=(\ZZ/2)^3$, all seven nonidentity elements are
involutions. Thus $t(G)=7$, and the formula recovers the
infinitesimal rigidity of the tangent bundle in
Section~\ref{sec:rigid}.

The quaternion group
$
Q_8=\{\pm1,\pm i,\pm j,\pm k\}
$
has exactly one involution, namely $-1$; the other six
nonidentity elements have order four.
Consequently, a free $Q_8$-action on such a threefold gives
\[
h^1(Y/Q_8,\End T_{Y/Q_8})
=\frac{16(7-1)}8=12.
\]
Beauville constructed such an action in
\cite[Theorem~1.1]{Beauville1999}; see also
\cite[Section~2]{Hua2011}.

For comparison, let $Y\subset\PP^4$ be a smooth quintic
threefold with a free action of a finite group $G$ by
projective transformations, and set $X=Y/G$.
Here $h^1(Y,\End T_Y)=224$, and the equation space $U_5$
is one-dimensional. Therefore
$\End U_5$ is the trivial one-dimensional representation,
and \eqref{eq:general-quotient-count} becomes
\[
h^1(X,\End T_X)=\frac{225}{|G|}-1.
\]
For free actions of groups of orders five and twenty-five,
the resulting dimensions are $44$ and $8$, respectively.

\section{An obstructed tangent bundle and its complete Kuranishi germ}
\label{sec:igusa-obstructions}

We prove Theorem~\ref{thm:igusa-kuranishi} on the classical Igusa quotient.
The finite flat cover provides harmonic representatives for all
endomorphism-valued Dolbeault classes.  Their products give both a
nonzero primary obstruction and the exact analytic Kuranishi equations.
Throughout this section, the complex structure of the threefold is fixed.

\subsection{The quotient and its linear holonomy}

Let $A=E_1\times E_2\times E_3$ be a product of elliptic curves,
and choose nonzero two-torsion points $\tau_i\in E_i[2]$.
All coordinate operations use the elliptic-curve group laws. For
example, one may take $E_i=\CC/(\ZZ+i\ZZ)$ and $\tau_i=\tfrac12$.

Define automorphisms of $A$ by
\begin{align}
\sigma(z_1,z_2,z_3)
&=(z_1+\tau_1,-z_2,-z_3),
\label{eq:igusa-sigma}\\
\tau(z_1,z_2,z_3)
&=(-z_1,z_2+\tau_2,-z_3+\tau_3).
\label{eq:igusa-tau}
\end{align}
This is Igusa's construction
\cite[p.~678]{Igusa1954}, in the form described in
\cite[Example~2.17]{OguisoSakurai2001}.
We verify the properties needed below.

The two-torsion relations imply
$
\sigma^2=\tau^2=1
$
and
\[
\sigma\tau(z_1,z_2,z_3)
=\tau\sigma(z_1,z_2,z_3)
=(-z_1+\tau_1,-z_2+\tau_2,z_3+\tau_3).
\]
These formulas give four distinct automorphisms, so
\[
G=\langle\sigma,\tau\rangle
=\{1,\sigma,\tau,\sigma\tau\}
\simeq(\mathbb Z/2)^2.
\]

The elements $\sigma$, $\tau$, and $\sigma\tau$ act by nonzero
translations on the first, second, and third coordinates, respectively.
They therefore have no fixed points, and
\[
p:A\longrightarrow Z:=A/G
\]
is a finite \'etale cover of degree four.
In particular, $Z$ is a smooth projective threefold.

The translation terms disappear upon differentiation:
\[
D\sigma=\operatorname{diag}(1,-1,-1),
\qquad
D\tau=\operatorname{diag}(-1,1,-1).
\]
Let $e_i=\partial/\partial z_i$ denote the three constant
tangent directions on $A$.
The action on the line spanned by $e_i$ is given by a
character $\chi_i:G\to\{\pm1\}$.
Listing its values on the generators $(\sigma,\tau)$,
we obtain
\begin{equation}\label{eq:igusa-characters}
\chi_1=(+,-),\qquad
\chi_2=(-,+),\qquad
\chi_3=(-,-).
\end{equation}
These are the three nontrivial characters of $G$.

A product of translation-invariant flat K\"ahler metrics
on the elliptic curves is preserved by $G$, and hence
descends to a flat K\"ahler metric on $Z$.
Its linear holonomy group is the group of four diagonal
matrices generated by $D\sigma$ and $D\tau$.
The three characters above describe its action on
the three tangent directions.

Since
$
\chi_1\chi_2\chi_3=1,
$
the holomorphic three-form
$
\omega=dz_1\wedge dz_2\wedge dz_3
$
is $G$-invariant. It therefore descends to a holomorphic
three-form on $Z$. This descended form is nowhere
vanishing because $p$ is \'etale, so
$
K_Z\simeq\mathcal O_Z.
$

To compute the low-degree cohomology of $\mathcal O_Z$,
put $\eta_i=d\overline z_i$.
On the product torus $A$, Dolbeault cohomology gives
\[
H^1(A,\mathcal O_A)
=\bigoplus_{i=1}^3\mathbb C[\eta_i],
\qquad
H^2(A,\mathcal O_A)
=\bigoplus_{i<j}\mathbb C[\eta_i\wedge\eta_j].
\]
The class $[\eta_i]$ has character $\chi_i$, since the
signs in \eqref{eq:igusa-characters} are real.
The class $[\eta_i\wedge\eta_j]$ has character
$\chi_i\chi_j$. Moreover,
$
\chi_1\chi_2=\chi_3,$ $
\chi_1\chi_3=\chi_2,$ $
\chi_2\chi_3=\chi_1.
$
Thus neither cohomology group contains a trivial
representation.
Taking invariants under the finite \'etale cover gives
\[
H^1(Z,\mathcal O_Z)
=H^1(A,\mathcal O_A)^G=0,
\qquad
H^2(Z,\mathcal O_Z)
=H^2(A,\mathcal O_A)^G=0.
\]

The decomposition
$
T_A=\mathcal O_Ae_1\oplus
    \mathcal O_Ae_2\oplus
    \mathcal O_Ae_3
$
is preserved by $G$.
Each summand therefore descends to a line bundle $L_i$
on $Z$, giving
\begin{equation}\label{eq:igusa-tangent-splitting}
T_Z=L_1\oplus L_2\oplus L_3.
\end{equation}
The transition functions of $L_i$ can be chosen among
the values of $\chi_i$, namely $\pm1$.
Consequently, each $L_i$ is unitary flat and satisfies
$
L_i^{\otimes2}\simeq\mathcal O_Z.
$

The $L_i$ are pairwise nonisomorphic. A homomorphism $L_i\to L_j$
pulls back to a constant $a$ on $A$, and equivariance requires
$\chi_j(g)a=a\chi_i(g)$ for every $g\in G$. Since the characters are
distinct,
\[
H^0(Z,\operatorname{Hom}(L_i,L_j))=0\qquad(i\ne j).
\]

For any ample polarization $H$, the relation
$L_i^{\otimes2}\simeq\mathcal O_Z$ implies
$
\mu_H(L_i)=c_1(L_i)\cdot H^2=0.
$
Each line bundle is stable, so
\eqref{eq:igusa-tangent-splitting} expresses $T_Z$
as a direct sum of stable bundles of the same slope.
Thus $T_Z$ is polystable for every polarization.
It is not stable, because each $L_i$ is a proper
subbundle of the same slope as $T_Z$.
Furthermore,
$
H^0(Z,\End T_Z)\simeq\mathbb C^3,
$
with the three factors acting by scalar multiplication
on the three summands. In particular, $T_Z$ is not simple.

Finally, the covering $p$ gives an exact sequence
\[
1\longrightarrow\pi_1(A)
\longrightarrow\pi_1(Z)
\longrightarrow G
\longrightarrow1.
\]
Since $\pi_1(A)\simeq\mathbb Z^6$, the fundamental
group of $Z$ contains $\mathbb Z^6$ as a subgroup
of index four and is therefore infinite.
\subsection{A finite-dimensional differential graded Lie model}

We study deformations of the holomorphic bundle $T_Z$ on the
fixed complex manifold $Z$.
Write $e_i=\partial/\partial z_i$ for the three constant
tangent directions on the covering torus $A$.
Let $E_{ij}$ be the matrix unit defined by
$
E_{ij}(e_j)=e_i,$ $
E_{ij}(e_\ell)=0\quad\text{if }\ell\ne j.$
Thus $E_{ij}E_{kl}=\delta_{jk}E_{il}$.
For matrix-valued differential forms, we combine matrix
multiplication with the exterior product:
\[
(\omega E_{ij})(\theta E_{kl})
=\delta_{jk}(\omega\wedge\theta)E_{il}.
\]

The Dolbeault differential graded Lie algebra associated
with $T_Z$ is
\[
\mathfrak g
=
\bigl(
A^{0,\bullet}(Z,\operatorname{End}T_Z),
\bar\partial,[\, ,\,]
\bigr).
\]
As in Section~\ref{sec:prelim}, a deformation
$\bar\partial_\alpha=\bar\partial+\alpha$ is integrable precisely when
\begin{equation}\label{eq:igusa-mc}
\bar\partial\alpha+\tfrac12[\alpha,\alpha]
=\bar\partial\alpha+\alpha^2=0,
\end{equation}
For the Dolbeault DGLA and its gauge action, see
\cite[Example~V.22 and Chapter~V]{Manetti2004}.

We now replace this infinite-dimensional DGLA by a
finite-dimensional subalgebra.
Equip $A$ with a product translation-invariant flat
K\"ahler metric and $T_A$ with the induced flat
Hermitian metric. These metrics are preserved by $G$
and descend to $Z$.
Since $p^*T_Z\simeq T_A$ is trivial, pullback identifies
$\mathfrak g$ with
$
(
A^{0,\bullet}(A)\otimes
\operatorname{Mat}_3(\mathbb C)
)^G.
$
Inside this complex, consider the invariant constant forms

\begin{equation}\label{eq:igusa-constant-dgla}
\mathfrak h^q
=
\left(
\bigwedge\nolimits^q
\langle\eta_1,\eta_2,\eta_3\rangle
\otimes\operatorname{Mat}_3(\mathbb C)
\right)^G.
\end{equation}
The Dolbeault differential vanishes on $\mathfrak h$.
Moreover, the product and bracket of invariant constant
forms are again invariant constant forms.
Thus $(\mathfrak h,0,[\, ,\,])$ is a finite-dimensional
DGLA subalgebra of $\mathfrak g$.

Harmonic forms on a flat complex torus have constant coefficients,
also for matrix-valued forms. Since the metric is $G$-invariant,
harmonic representatives of invariant cohomology classes are invariant.
The inclusion therefore induces isomorphisms
\[
\mathfrak h^q
\xrightarrow{\sim}
H^q(\mathfrak g)
\simeq H^q(Z,\operatorname{End}T_Z)
\]
for every $q$.
It follows that
\begin{equation}\label{eq:igusa-formality}
(\mathfrak h,0,[\, ,\,])
\hookrightarrow
(\mathfrak g,\bar\partial,[\, ,\,])
\end{equation}
is a quasi-isomorphism of DGLAs.
In particular, $\mathfrak g$ is formal: it is
quasi-isomorphic to its cohomology equipped with the
induced bracket and zero differential.

It remains to determine the invariant constant forms.
The action on endomorphisms is by conjugation, so
$E_{ij}$ has character
$
\chi_i\chi_j^{-1}=\chi_i\chi_j.
$
The diagonal matrices $E_{ii}$ have trivial character.
If $i\ne j$ and $\{i,j,k\}=\{1,2,3\}$, then $E_{ij}$
has character $\chi_k$.
More generally, the character of
$
\eta_{a_1}\wedge\cdots\wedge\eta_{a_q}E_{ij}
$
is
$
\chi_{a_1}\cdots\chi_{a_q}\chi_i\chi_j.
$
Such a form is invariant exactly when this product is
the trivial character.

In degree zero, only the diagonal matrices are invariant:
\begin{equation}\label{eq:igusa-h0}
\mathfrak h^0
=\langle E_{11},E_{22},E_{33}\rangle.
\end{equation}
In degree one, the character of $\eta_k$ must cancel
the character of $E_{ij}$.
This gives
\begin{equation}\label{eq:igusa-h1}
\begin{aligned}
\mathfrak h^1=\langle&
\eta_3E_{12},\eta_1E_{23},\eta_2E_{31},\eta_3E_{21},\eta_2E_{13},\eta_1E_{32}
\rangle.
\end{aligned}
\end{equation}

For degree two, the forms
$\eta_1\wedge\eta_2$,
$\eta_1\wedge\eta_3$, and
$\eta_2\wedge\eta_3$
have characters $\chi_3$, $\chi_2$, and $\chi_1$,
respectively. Therefore
\[
\begin{aligned}
\mathfrak h^2=\langle&
(\eta_1\wedge\eta_2)E_{12},
(\eta_1\wedge\eta_2)E_{21},(\eta_1\wedge\eta_3)E_{13},\\
&(\eta_1\wedge\eta_3)E_{31},(\eta_2\wedge\eta_3)E_{23},
(\eta_2\wedge\eta_3)E_{32}
\rangle.
\end{aligned}
\]
Finally, $\eta_1\wedge\eta_2\wedge\eta_3$ has trivial
character, so
\[
\mathfrak h^3
=
(\eta_1\wedge\eta_2\wedge\eta_3)
\langle E_{11},E_{22},E_{33}\rangle.
\]
The quasi-isomorphism
\eqref{eq:igusa-formality} now gives
\[
h^q(Z,\operatorname{End}T_Z)
=
\begin{cases}
3,&q=0,3,\\
6,&q=1,2.
\end{cases}
\]

The Maurer--Cartan equation in $\mathfrak h$ is $\alpha^2=0$.
The bracket is nonzero; for example,
\[
[\eta_3E_{12},\eta_1E_{23}]
=(\eta_3\wedge\eta_1)E_{13}\ne0.
\]

\subsection{The quadratic equations and analytic completeness}

Using the basis of \eqref{eq:igusa-h1}, write a general
element of $\mathfrak h^1$ as
\begin{equation}\label{eq:igusa-alpha}
\alpha=
\begin{pmatrix}
0&a\eta_3&e\eta_2\\
d\eta_3&0&b\eta_1\\
c\eta_2&f\eta_1&0
\end{pmatrix},
\end{equation}
where $a,b,c,d,e,f\in\mathbb C$.
Since $\alpha$ has constant coefficients,
$\bar\partial\alpha=0$, so its Maurer--Cartan equation
is simply $\alpha^2=0$.

The matrix product uses the exterior product of the
form coefficients:
\[
(\alpha^2)_{ij}
=\sum_k\alpha_{ik}\wedge\alpha_{kj}.
\]
For example,
$
(\alpha^2)_{13}
=(a\eta_3)\wedge(b\eta_1)
=-ab\,\eta_1\wedge\eta_3.
$
Computing all entries gives
\begin{align}\label{eq:igusa-square}
\alpha^2={}&
bc\,(\eta_1\wedge\eta_2)E_{21}
-ef\,(\eta_1\wedge\eta_2)E_{12}\notag\\
&+df\,(\eta_1\wedge\eta_3)E_{31}
-ab\,(\eta_1\wedge\eta_3)E_{13}\notag\\
&+ac\,(\eta_2\wedge\eta_3)E_{32}
-de\,(\eta_2\wedge\eta_3)E_{23}.
\end{align}
All diagonal entries vanish. For instance,
$
(\alpha^2)_{11}
=ad\,\eta_3\wedge\eta_3
+ec\,\eta_2\wedge\eta_2=0.
$
The six matrix-valued two-forms in
\eqref{eq:igusa-square} form a basis of $\mathfrak h^2$.
Their coefficients must therefore vanish separately,
and we obtain
\[
\alpha^2=0
\quad\Longleftrightarrow\quad
ab=bc=ca=de=ef=fd=0.
\]
Thus the constant solutions are parametrized by
\[
B:=V(ab,bc,ca,de,ef,fd)\subset\mathbb C^6,
\]
the space appearing in \eqref{eq:igusa-kuranishi}.

To identify $(B,0)$ with the full analytic semiuniversal base, we
use the Kuranishi fixed-point equation.

Let
$
\Delta_{\bar\partial}
=\bar\partial\bar\partial^*
+\bar\partial^*\bar\partial
$
be the Dolbeault Laplacian on endomorphism-valued forms.
Write $H$ for harmonic projection and $\mathcal G$
for the Green operator, normalized to vanish on
harmonic forms. Set
$
h=\bar\partial^*\mathcal G.
$
The Hodge decomposition gives
$
\operatorname{id}=H+\bar\partial h+h\bar\partial.
$
The analytic Kuranishi construction associates to
each sufficiently small harmonic parameter
$u\in\mathfrak h^1$ a form $\beta(u)$ satisfying
\begin{equation}\label{eq:igusa-kuranishi-fixedpoint}
\beta
=u-\tfrac12h[\beta,\beta].
\end{equation}
This solution satisfies the Kuranishi gauge
conditions
$
H\beta(u)=u,$ $
\bar\partial^*\beta(u)=0.
$
Its integrability is equivalent to the vanishing of
the Kuranishi obstruction map
\[
\kappa(u)
=\tfrac12H[\beta(u),\beta(u)].
\]
For the vector-bundle deformation theory and the
analytic DGLA formulation of the Kuranishi
construction, see
\cite[Appendix~A]{Huybrechts1995}
and \cite{GoldmanMillson1990}.

In the present example, harmonic forms have constant
coefficients, and their bracket again has constant
coefficients. Thus $[u,u]$ is harmonic, which implies
$
\mathcal G[u,u]=0,
$ $
h[u,u]=0.
$
It follows that $\beta=u$ solves
\eqref{eq:igusa-kuranishi-fixedpoint} for every
sufficiently small $u$.
By uniqueness, this is exactly the Kuranishi solution:
$
\beta(u)=u.
$
Consequently,
\[
\kappa(u)
=\tfrac12H[u,u]
=\tfrac12[u,u]
=u^2.
\]
Thus the obstruction map is exactly quadratic, and
$\bar\partial+u$ is integrable precisely when $u^2=0$.

The family of operators $\bar\partial+\alpha$ over
$(B,0)$ is therefore the analytic Kuranishi family.
The Kuranishi theorem makes this family semiuniversal:
every sufficiently small deformation is locally
obtained from it by base change, and its
Kodaira--Spencer map is an isomorphism.
With our chosen basis, that map is the identity
\[
T_0B\simeq\mathfrak h^1
\simeq H^1(Z,\End T_Z).
\]

The equations $ab=bc=ca=0$ define the three coordinate axes in
$\CC^3$, and the same holds for $de=ef=fd=0$. Hence
\[
B=V(ab,bc,ca)\times V(de,ef,fd).
\]
To check the analytic structure as well, work in the convergent
power-series ring $\CC\{a,b,c,d,e,f\}$. For
$x\in\{a,b,c\}$ and $y\in\{d,e,f\}$, let $\mathfrak p_{xy}$ be the
ideal generated by the four variables other than $x,y$. The monomial
identity
\[
(ab,bc,ca,de,ef,fd)
=\bigcap_{\substack{x\in\{a,b,c\}\\y\in\{d,e,f\}}}
\mathfrak p_{xy}
\]
expresses the defining ideal as an intersection of nine coordinate
prime ideals. Thus $B$ is reduced, and its nine irreducible components
are precisely the two-dimensional planes retaining one variable from
each triple.

On the other hand, all its defining equations are
quadratic, so their differentials vanish at the
origin. Therefore
$
T_0B\simeq\mathbb C^6.
$
Thus the semiuniversal base has nine
two-dimensional branches, while its Zariski tangent
space at the origin has dimension six.

\subsection{A single obstructed class and the moduli interpretation}

The nonsmoothness can also be seen without the full equations.
Consider the class represented by
\[
 \xi=\eta_3E_{12}+\eta_1E_{23}\in\mathfrak h^1.
\]
Since $E_{12}E_{23}=E_{13}$ and $E_{23}E_{12}=0$, its Yoneda square is
\begin{equation}\label{eq:igusa-explicit-obstruction}
 \xi^2=\eta_3\wedge\eta_1 E_{13}\ne0
 \quad\text{in }H^2(Z,\operatorname{End}T_{Z}).
\end{equation}
Nonvanishing follows either from the constant harmonic description
or by pulling back to $A$, where the indicated class is manifestly
nonzero.  The first-order operator
$\bar\partial+t\xi$ over $\mathbb C[t]/(t^2)$ could extend to
$\mathbb C[t]/(t^3)$ only if a correction $t^2\gamma$ satisfied
$\bar\partial\gamma=-\xi^2$.  Equation
\eqref{eq:igusa-explicit-obstruction} rules this out.

Finally, \eqref{eq:igusa-h0} gives
$
 \operatorname{Aut}(T_{Z})=(\mathbb C^*)^3.
$
An automorphism $\operatorname{diag}(\lambda_1,\lambda_2,\lambda_3)$
acts on the coefficient of $\eta_kE_{ij}$ by
$\lambda_i\lambda_j^{-1}$.  The diagonal scalar subgroup acts
trivially.  A coarse moduli space involves further identifications
by automorphisms (and, in a semistable moduli problem,
$S$-equivalence).  Accordingly, \eqref{eq:igusa-kuranishi} is a
semiuniversal deformation base, not an asserted local equation for
such a coarse moduli space.  Its singularity, and especially the
nonzero square \eqref{eq:igusa-explicit-obstruction}, arise before
any such quotient.  This completes the proof of
Theorem~\ref{thm:igusa-kuranishi}.

\begin{remark}\label{rem:aspinwall}
Aspinwall \cite[Section~7, equations~(51)--(57)]{Aspinwall2014} computes a
nonzero Yoneda product and a cubic potential $W=XYZ$ in the classical
local deformation problem associated with a crepant resolution of
$\CC^3/\frac17(1,2,4)$.  Its critical equations are
$XY=XZ=YZ=0$, the equations of three coordinate axes.  Thus this type of
quadratic obstruction already occurs in the literature.  The calculation
above realizes a product of two such germs as the complete semiuniversal
base of a tangent bundle on a compact threefold; it does not rely on a
globalization of the noncompact example.  The tangent bundle here remains
polystable rather than stable.
\end{remark}

\section{Fundamental groups and the scope of the examples}
\label{sec:infinite}

The rigid example has finite fundamental group, whereas the Igusa
quotient has infinite fundamental group. Under Convention~\ref{conv:cy},
this distinction is necessary: infinite fundamental group rules out
infinitesimal rigidity of the tangent bundle.

\begin{proposition}\label{prop:infinite}
If $X$ satisfies Convention~\ref{conv:cy} and $\pi_1(X)$ is
infinite, then
\[
H^1(X,\End T_X)\ne0.
\]
Equivalently, infinitesimal rigidity of $T_X$ forces
$\pi_1(X)$ to be finite.
\end{proposition}

\begin{proof}
By the Beauville--Bogomolov decomposition, an infinite fundamental
group implies that $X$ has a finite \'etale cover by either an
abelian threefold $A$ or a product $S\times E$, where $S$ is a
K3 surface and $E$ is an elliptic curve;
see \cite{Beauville1983,OguisoSakurai2001}.

We may take this cover to be Galois. Indeed, passing to the Galois
closure gives a connected finite \'etale cover of the original
covering space. Such a cover of an abelian threefold is again an
abelian threefold. Such a cover of $S\times E$ is of the form
$S\times E'$ for an elliptic curve $E'$, because $S$ is simply
connected.

Write the resulting cover as
$
p:\widetilde X\to X=\widetilde X/G,
$
where $G$ acts freely. Lemma~\ref{lem:descent} gives
$
H^1(X,\End T_X)
\simeq H^1(\widetilde X,\End T_{\widetilde X})^G.
$
It therefore suffices to construct a nonzero $G$-invariant class
on the covering space.

We use compatible left actions throughout: pushforward on vector
fields and inverse pullback on differential forms and their
cohomology.

\medskip
\noindent\emph{Case 1: $\widetilde X=A$ is an abelian threefold.}
Set
$
V=H^0(A,T_A).
$
The tangent bundle is equivariantly trivial:
$
T_A\simeq\mathcal O_A\otimes V.
$
Every $g\in G$ acts affinely,
$
g(a)=L_g a+t_g,
$
and its action on $V$ is given by the linear part $L_g$.

For $g\ne1$, the matrix $L_g$ must have eigenvalue $1$.
Otherwise, $1-L_g:A\to A$ would be
surjective. The equation
$
(1-L_g)a=t_g
$
would then have a solution, giving a fixed point of $g$.
This contradicts the freeness of the action.
The identity element also has eigenvalue $1$.

A nowhere vanishing holomorphic three-form on $X$ pulls back to
a $G$-invariant three-form on $A$, so
$
\det L_g=1.
$
Since $G$ is finite, the eigenvalues of $L_g$ are roots of unity.
They therefore have the form
$
1,\lambda,\lambda^{-1}.
$
In particular,
$
\operatorname{tr}(g\mid V)
=1+\lambda+\lambda^{-1}\in\mathbb R.
$
Thus $V$ and $V^*$ have the same character, and hence
$
V\simeq V^*
$
as $G$-representations.

With our action convention,
$
H^1(A,\mathcal O_A)\simeq\overline V^{\,*}.
$
Choosing a $G$-invariant Hermitian inner product identifies
$\overline V^{\,*}$ with $V$. Consequently,
\[
\begin{aligned}
H^1(A,\End T_A)
&\simeq H^1(A,\mathcal O_A)\otimes V\otimes V^*\simeq V^{\otimes3}.
\end{aligned}
\]
Antisymmetrization gives a $G$-equivariant inclusion
$
\Lambda^3V\hookrightarrow V^{\otimes3}.
$
Since $\det L_g=1$ for every $g$, the one-dimensional
representation $\Lambda^3V$ is trivial. Hence
$
H^1(A,\End T_A)^G\ne0,
$
which proves the claim in this case.

\medskip
\noindent\emph{Case 2: $\widetilde X=S\times E$.}
Every deck transformation splits as
\[
g=(f_g,h_g),
\qquad
f_g\in\operatorname{Aut}(S),\quad
h_g\in\operatorname{Aut}(E).
\]
Indeed, every holomorphic map $S\to E$ is constant, so the
$E$-component depends only on the $E$-coordinate.
The resulting family of automorphisms of $S$ is constant because
$\operatorname{Aut}(S)$ is discrete and $E$ is connected.

Choose nonzero forms
$
\sigma\in H^0(S,K_S),
$ $
\theta\in H^0(E,\Omega_E^1),
$
and write
$
g\cdot\sigma=a(g)\sigma,
$ $
g\cdot\theta=b(g)\theta.
$
The product three-form is $G$-invariant, so
$
a(g)b(g)=1.
$

We claim that both characters take the same values in
$\{1,-1\}$.

If $b(g)=1$, then $a(g)=1$.
If $b(g)\ne1$, the affine automorphism $h_g$ has nonidentity
linear part and therefore has a fixed point on $E$.
Since $g$ acts freely on $S\times E$, the automorphism $f_g$
must then have no fixed point on $S$.
For a K3 surface,
$
H^0(S,\mathcal O_S)=\mathbb C,$ $
H^1(S,\mathcal O_S)=0,$ $
H^2(S,\mathcal O_S)\simeq H^0(S,K_S)^*.
$
The fixed-point-free holomorphic Lefschetz formula therefore gives
$
0=1+a(g)^{-1}.
$
Thus $a(g)=b(g)=-1$.

It follows that there is a common sign character
$
\varepsilon:G\to\{1,-1\}
$
such that
$
g\cdot\sigma=\varepsilon(g)\sigma,
$ $
g\cdot\theta=\varepsilon(g)\theta.
$
Average a K\"ahler form on $S$ over $G$. Its cohomology class gives
a nonzero invariant element
$
\omega\in H^1(S,\Omega_S^1)^G.
$
Contraction with the nowhere degenerate two-form $\sigma$ gives
an isomorphism
$
C_\sigma:T_S\xrightarrow{\sim}\Omega_S^1,
$ $
v\longmapsto\iota_v\sigma.
$
Let
$
\beta=C_\sigma^{-1}(\omega)\in H^1(S,T_S).
$
Then $\beta\ne0$. Naturality of contraction gives
$
g\cdot(C_\sigma\beta)
=C_{g\cdot\sigma}(g\cdot\beta)
=\varepsilon(g)C_\sigma(g\cdot\beta).
$
Since $\omega$ is invariant and $\varepsilon(g)^2=1$, this implies
$
g\cdot\beta=\varepsilon(g)\beta.
$
Thus $\beta$ and $\theta$ transform by the same sign, and
$
\beta\otimes\theta
$
is a nonzero invariant tensor.

To place this tensor in the required cohomology group, observe that
$
T_{S\times E}
=\operatorname{pr}_S^*T_S\oplus\operatorname{pr}_E^*T_E.
$
Hence $\End T_{S\times E}$ has the $G$-invariant direct summand
$
\mathcal Hom\bigl(
\operatorname{pr}_E^*T_E,\operatorname{pr}_S^*T_S
\bigr)
=T_S\boxtimes\Omega_E^1.
$
The K\"unneth formula, together with $H^0(S,T_S)=0$, gives
$
H^1(S\times E,T_S\boxtimes\Omega_E^1)
\simeq H^1(S,T_S)\otimes H^0(E,\Omega_E^1).
$
Therefore $\beta\otimes\theta$ determines a nonzero class in
$
H^1(S\times E,\End T_{S\times E})^G.
$
Lemma~\ref{lem:descent} now proves
$H^1(X,\End T_X)\ne0$.
\end{proof}

The finite fundamental group need not be trivial, as
Theorem~\ref{thm:rigid} shows. That theorem gives a stable rigid tangent
bundle on a nonsimply connected threefold; its simply connected cover
has a nonrigid tangent bundle. The Igusa example has an obstructed
polystable tangent bundle, but it is a sum of flat line bundles.
Neither construction addresses the corresponding question for a
simply connected threefold with stable tangent bundle.

\bibliographystyle{amsalpha}
\bibliography{references}
\end{document}